\documentclass{mcom-l}

\newtheorem{theorem}{Theorem}[section]
\newtheorem{lemma}[theorem]{Lemma}

\theoremstyle{definition}

\theoremstyle{remark}
\newtheorem{remark}[theorem]{Remark}

\numberwithin{equation}{section}
\usepackage{graphicx}
\usepackage{amssymb,amsmath,amsbsy,mathrsfs,eucal,amsfonts}

\usepackage{multirow}
\usepackage{xcolor}
\definecolor{blck}{rgb}{0,0,0}

\definecolor{darkred}{rgb}{.6,.1,0}

\definecolor{blue}{rgb}{0,0,1}

\definecolor{red}{rgb}{1,0,0}

\usepackage{tikz}

\newcommand{\jmp}[1]{[\![#1]\!]}                     

\newcommand\rev[1]{\textcolor{blck}{#1}}

\newcommand{\Eh}{\mathscr{E}_h}

\def\calF{{\mathcal F}}

\def\calT{{\mathcal T}}

\def\calF{{\mathcal F}}

\begin{document}

\title{Optimally convergent hybridizable discontinuous Galerkin method for fifth-order Korteweg-de Vries type equations}

\author{Bo Dong}
\address{Department of
Mathematics, University of Massachusetts Dartmouth, 285 Old Westport Road, North Dartmouth, MA 02747, USA}
\curraddr{} \email{bdong@umassd.edu}
\thanks{The research of the third author was supported by
        the National Science Foundation (grant DMS-1419029). Corresponding author: B. Dong}

\author{Jiahua Jiang}
\address{Department of
Mathematics, University of Massachusetts Dartmouth, 285 Old Westport Road, North Dartmouth, MA 02747, USA}
\curraddr{} \email{jjiang@umassd.edu}
\thanks{}

\author{Yanlai Chen}
\address{Department of
Mathematics, University of Massachusetts Dartmouth, 285 Old Westport Road, North Dartmouth, MA 02747, USA}
\curraddr{} \email{yanlai.chen@umassd.edu}
\thanks{}


\date{}

\dedicatory{}

\begin{abstract}
We develop and analyze the first hybridizable discontinuous Galerkin (HDG) method for solving fifth-order Korteweg-de Vries (KdV) type equations. We show that the semi-discrete scheme is stable with proper choices of the stabilization functions in the numerical traces. For the linearized fifth-order equations, we prove that the approximations to the exact solution and its four spatial derivatives as well as its time derivative all have optimal convergence rates. The numerical experiments, demonstrating optimal convergence rates for both the linear and nonlinear equations, validate our theoretical findings.
\end{abstract}

\maketitle

\pagestyle{myheadings} \thispagestyle{plain} \markboth{B. DONG, J. JIANG, AND Y. CHEN}
{HDG for fifth-order KdV}

\section{Introduction}

Natural phenomena modeled by nonlinear partial differential equations are ubiquitous, appearing in numerous areas of science and engineering such as plasma physics, fluid dynamics, nonlinear optics, quantum mechanics, mathematical biology and chemical kinetics etc.
Of particular interest to us is the class of nonlinear, dispersive partial differential equations called Korteweg-de Vries (KdV) equations. They model  the evolution of long, one-dimensional waves such as the shallow-water waves with weakly nonlinear restoring forces and the long internal waves in a density-stratified ocean etc.

Due to their pervasiveness in applications and theoretical studies, it is critical to devise efficient numerical schemes for KdV equations that have mathematically provable stability and accuracy.
This paper is a continuation of our systematic effort \cite{ChenCockburnDongHDG16, Dong17} toward that end. 
Indeed, we focus on the following fifth-order KdV type equation:
\begin{equation}
\label{eq:prob}
u_t + \alpha u_{xxx} + \beta u_{xxxxx} + \frac{\partial}{\partial x} F(u, u_x, u_{xx}) = f
\end{equation}
for $x\in \Omega:=(0, L)$ and $t\in (0, T]$ with the initial condition
\[
u(x,0) = u_0(x)
\]
and periodic boundary condition.
Here the coefficients $\alpha$, $\beta$ are real numbers with $\beta < 0$, $f\in L^2(\Omega)$ and $F$ is a smooth function.\\

The fifth-order KdV equation \eqref{eq:prob} is the generic model for studying shallow water waves having surface tension and acoustic waves in plasma \cite{AkylasYang95, HunterScheurle88, Nagashima79, KakutaniOno69,Kawahara72}. 
One characteristic feature is the existence of solitary wave solutions that do not vanish at infinity maintaining their shapes \cite{Beale91, Body91}. The first fifth-order KdV type equation introduced in literature is the Kawahara equation \cite{Kawahara72,Kato12} for which $F(u,  u_x, u_{xx})=Ku^2$ where $K$ is a nonzero constant. Later, other forms of $F$ with more \rev{complicated} nonlinearities appeared due to the need to capture solitary wave solutions. In applications, the \rev{form of $F$ used most often} is $F(u, u_x, u_{xx})= K u^{n+1}$  where $n\ge 1$ an integer.
The scheme developed in this paper can be extended to general form of $F$ in a trivial fashion. However, for simplicity we focus on the case where $F$ is a function of $u$ only.\\

The {\em acute challenge} in numerically resolving these equations are two-fold. First and foremost is the need for accurate long-time integration. To capture the structure of the characteristic soliton solutions over a long time-span, it is imperative for the scheme to have very low phase error which is a trademark of high-order methods. The second challenge originates from the demand for preservation of physical quantities such as mass and Hamiltonian.
Discontinuous Galerkin (DG) method emerges as a natural choice for this scenario thanks to that it is easily implementable to be high order, its capability of handling nonuniform meshes (e.g. for a moving mesh following the soliton) with variable degrees, and its flexibility in design allowing for purposely constructing the numerical fluxes to achieve preservation of particular physical quantities \cite{ChenCockburnDongDG16}. 
Indeed, in 2002,  the first local DG method was introduced for third- and fifth-order KdV equations \cite{YanShu02}. It was further studied in \cite{XuShu04, XuShu07, HuffordXing14, GuoXu15} and was shown to have approximate solutions with optimal convergence rates for linearized KdV equations \cite{XuShu12}. In \cite{ChengShu08}, another DG method was designed for KdV type equations based on
repeated integration by parts. This method has optimal
convergence rates for linearized KdV equations, but needs to use at least second-order polynomials for third-order equations and fourth-order polynomials for fifth-order equations. In \cite{BonaChenKarakashianXing13,ChenCockburnDongDG16,KarakashianXing16}, conservative DG methods were constructed for third-order KdV type equations to
preserve quantities such as the mass and the $L^2$-norm of the solutions. These methods have optimal convergence rates when approximate solutions have even polynomial degrees and suboptimal convergence rates when the polynomial degrees are odd.

Traditional DG methods, despite their prominent features, 
were criticized for having too many degrees of freedom and for not being efficiently implementable. As a response to these criticisms, the hybridizable discontinuous Galerkin (HDG) methods were designed. Their unique feature is that the globally coupled degrees of freedom are only those on the element interfaces. Hence, they are more advantageous than traditional DG methods for solving stationary equations or time-dependent problems that require global solvers. In addition to the dimension reduction, they usually produce optimally convergent approximations for both the primal and dual variables, a feature notably lacking for the traditional DG formulation. 
The dimension reduction in the globally coupled degrees of freedom has a remarkable consequence for the one-dimensional setting of the KdV equations.
In fact, {\em the size and bandwidth of the global system to be numerically solved are independent of the polynomial degrees in the finite element space.}
This feature renders the HDG method particular attractive in the setting of this paper since the high-order accuracy is indispensable for accurate long time integration.

The idea of utilizing HDG method for KdV type equations is not new. In \cite{ChenCockburnDongHDG16}, we developed the first class of HDG methods for solving stationary third-order linear equations and proved super-convergence results. In \cite{Dong17}, the HDG method is extended to the (nonlinear) third-order  KdV type equations, attaining rigorous optimal convergence of  approximate solutions to the primary variable and its two derivatives for the linearized equations.
This paper amounts to {\em the first} attempt to design HDG and analyze it for the fifth-order KdV type equations to the best of our knowledge. 
Indeed, we design a semi-discrete HDG scheme for the general nonlinear equation \eqref{eq:prob} and prove its $L^2$-stability. For the fully discrete scheme, we adopt an implicit time-marching approach to \rev{mitigate the severe time-step restriction}, and at each time step the HDG method to solve the resulting stationary fifth-order equation.
\rev{To ensure compatibility at the first step and avoid loss of accuracy during the time-marching}, the HDG solver for stationary fifth-order equations is also used to compute the initial approximations for the time discretization. For stationary fifth-order linear equations, we prove that the HDG approximations of the primary variable with its four spatial derivatives 
superconverge to an HDG projection of the solutions with order $k+2$ and the numerical traces superconverge with the order $2k+1$ when the polynomial degree of the approximate solutions $k>0$. This is obtained by using a special duality argument and Green's functions. 
For the time-dependent case, we 
combine several energy identities and  prove that the semi-discrete form has optimal convergence rates for the approximations to the primary variable as well as its four spatial derivatives and the time derivative. 

The remainder of the paper is structured as follows. In Section \ref{sec:mainresults}, we introduce the central theme of this paper, a new HDG method for the KdV equation \eqref{eq:prob}, and present the associated stability analysis. A priori error analyses for the linearized KdV equations are presented in Section \ref{sec:erroranalysis}.
We omit some details of the proof and postpone them until the \rev{appendices}. The numerical results collected in Section \ref{sec:numericaltest} corroborate our theoretical predictions and show optimal convergence of the scheme for both linear and nonlinear equations.

\section{The HDG scheme and its stability}
\label{sec:mainresults}

In this section, we define the HDG method before stating and proving the theoretical result on the stability of the scheme.

\subsection{The HDG scheme}

\subsubsection{Notation}
We first introduce a partition of the computational domain,
\[
\calT_h=\left\{I_i:=(x_{i-1}, x_i): i = 1, \dots, N\right\},
\]
with $0=x_0<x_1<\cdots<x_{N-1}<x_N=L$.
We use $\partial\calT_h:=\{ \partial I_i: i=1,\dots,N\}$ to denote  the boundaries of the intervals, and $\Eh:=\{x_i\}_{i=0}^N$ to denote the set of all nodes. We set $h_i = x_i - x_{i-1}$ and $h:=\max\{h_i,\; i=1,\dots,N\}$.

For any function $\zeta\in L^2(\partial\calT_h)$, we denote its values on $\partial I_i:=\{x^+_{i-1}, x^-_i\}$ by $\zeta(x_{i-1}^+)$ (or simply $\zeta^+_{i-1}$) and $\zeta(x_i^-)$ (or simply $\zeta^-_i$). Note that $\zeta(x_{i}^+)$ does not have to be equal to $\zeta(x_i^-)$. On the other hand, for any function $\eta\in L^2(\Eh)$, its value at $x_i$, $\eta(x_i)$ (or simply $\eta_i$) is uniquely defined. We denote by $(\varphi, v)_{I_i}$ the integral of the product of $\varphi$ and $v$ on the interval $I_i$, and $\langle \varphi,v n\rangle_{\partial I_i}$ denotes that on the boundary which, in the one-dimensional case, simply becomes $\varphi(x_i^-)v(x_i^-)n(x_i^-) +\varphi(x_{i-1}^+)v(x_{i-1}^+)n(x_{i-1}^+)$. Here $n$ denotes the outward unit normal to $I_i$, that is $n(x_{i-1}^+):=-1$ and $n(x_i^-):=1$. Finally, we will need the finite element space which, when restricted to any particular element $I_j$, is the space of polynomials of degree at most $k$:
\begin{equation*}
{W}_h^k = \{\omega \in L^2(\mathcal{T}_h): \quad  \omega|{{_{I_i}}} \in {P}_{k}(I_i) \quad\forall\; i=1, \cdots, N\}.
\end{equation*}

\subsubsection{Spatial discretization}
To define the semi-discretization of the equation \eqref{eq:prob}, we first rewrite the time-dependent fifth-order equation as the following first-order system:
\begin{subequations}
\begin{alignat}{1}
\label{eq:system}
&q - u_x \,=\, 0, \quad p - q_x\,=\, 0,\quad r - p_x \,=\, 0, \quad s - r_x \,=\, 0, \quad u_t+ \calF(s,p,u)_x \,=\,f,
\intertext{where $\calF(s,p,u) = \alpha p + \beta s + F(u)$, and the boundary conditions are} 
\label{eq:BC_P}
&\qquad\qquad\omega(0, t)=\omega(L, t)\quad \mbox{ for } \omega=u, q, p, r, \mathcal{F}. 
\end{alignat}
\end{subequations}

The HDG scheme provides approximations
\[
(u_h, q_h, p_h, r_h, s_h, \widehat{u}_h, \widehat{q}_h, \widehat{p}_h, \widehat{r}_h, \widehat{s}_h) \in \left[W_h^{k}\right]^5 \times \left[ L^2(\partial\calT_h) \right]^5
\]
to
\[
(u|_\Omega,q|_\Omega,p|_\Omega, r|_\Omega, s|_\Omega, u|_{\Eh},q|_{\Eh},p|_{\Eh}, r|_{\Eh},s|_{\Eh}).
\]
Indeed, to determine these approximations, assuming that we are given the boundary values  $\{\widehat{u}_{hi}\}_{i=0}^N$, $\{\widehat{q}_{hi}\}_{i=0}^N$, and $\{\widehat{p}^{\;-}_{hi}\}_{i=1}^N$ which are the only globally coupled unknowns, 
we solve the equation  \eqref{eq:system} locally on each element by adopting a Galerkin method. More specifically, on the element $I_i$, we give $f$ and the boundary data  {$\widehat{u}_{h\, i-1}, \widehat{u}_{h\,i}$, {$\widehat{q}_{h\, i-1}, \widehat{q}_{h\,i}$, and $\widehat{p}^{\; -}_{h\, i}$} and take $(u_h, q_h, p_h, r_h, s_h) \in \left[P_{k}(I_i)\right]^5$ to be the solution of the equations
\begin{subequations}
\label{eq:scheme},
\begin{alignat}{2}
\label{eq:scheme1}
({q}_h,{v})_{I_i} + (u_h,v_x)_{I_i} - \langle \widehat{u}_h,{v}n \rangle_{\partial I_i} & = 0, \\
\label{eq:scheme2}
({p}_h,{z})_{I_i} + (q_h,z_x)_{I_i} - \langle \widehat{q}_h,{z}n \rangle_{\partial I_i} & = 0,\\
\label{eq:scheme3}
({r}_h,{w})_{I_i} + (p_h,w_x)_{I_i} - \langle \widehat{p}_h,{w}n \rangle_{\partial I_i} & = 0, \\
\label{eq:scheme4}
({s}_h,{\phi})_{I_i} + (r_h,\phi_x)_{I_i} - \langle \widehat{r}_h,{\phi}n \rangle_{\partial I_i} & = 0,\\
\label{eq:scheme5}
({u}_{ht},{\psi})_{I_i} - (\calF(s_h,p_h,u_h), \psi_x)_{I_i} + \langle \widehat{\calF}_h, {\psi}n \rangle_{\partial I_i} & = (f, \psi)_{I_i},
\end{alignat}
\end{subequations}
for all $(v, z, w, \phi, \psi)\,\in\,(P_k(I_i))^5$. 
We close the local system by defining the remaining  numerical traces as follows
\begin{subequations}
\label{eq:fluxes}
\begin{alignat}{2}
\label{eq:fluxp+}
\widehat{p}_h^{\;+} & = p_h^+ + \tau_{pu}^+\,( {\widehat{u}_{h}} - u_h^+)\,n^+ + \tau_{pq}^+ (\widehat{q}_h - q_h^+)\,n^+ &\mbox{ at } x_{i-1}^+,\\
\label{eq:fluxr+}
\widehat{r}_h^{\;+} & = r_h^+ + \tau_{ru}^+\,( {\widehat{u}_{h}} - u_h^+)\,n^+ + \tau_{rq}^+ (\widehat{q}_h - q_h^+)\,n^+ &\mbox{ at } x_{i-1}^+,\\
\label{eq:fluxs+}
\widehat{s}_h^{\;+} & = s_h^+ + \tau_{su}^+\,( {\widehat{u}_{h}} - u_h^+)\,n^+ + \tau_{sq}^+ (\widehat{q}_h - q_h^+)\,n^+ &\mbox{ at } x_{i-1}^+,\\
\label{eq:fluxr-}
\widehat{r}_h^{\;-} & = r_h^- + \tau_{ru}^-\,( {\widehat{u}_{h}} - u_h^-)\,n^- + \tau_{rq}^- (\widehat{q}_h - q_h^-)\,n^- + \tau_{rp}^- (\widehat{p}_h^{\;-} - p_h^-)\,n^- &\mbox{ at } x_i^-,\\
\label{eq:fluxs-}
\widehat{s}_h^{\;-} & = s_h^- + \tau_{su}^-\,( {\widehat{u}_{h}} - u_h^-)\,n^- + \tau_{sq}^- (\widehat{q}_h - q_h^-)\,n^- + \tau_{sp}^- (\widehat{p}_h^{\;-} - p_h^-)\,n^- &\mbox{ at } x_i^-,\\
\label{eq:fluxF}
\widehat{\calF}_h & = \calF(\widehat{s}_h, \widehat{p}_h, \widehat{u}_h) - \tau_\calF(\widehat{u}_h, u_h) (\widehat{u}_h - u_h) n& \mbox{ at } x_i^-, x_{i-1}^+.
\end{alignat}
\end{subequations} 
Obviously, \eqref{eq:fluxp+}--\eqref{eq:fluxs+} are defined at $x_0, \cdots, x_{N-1}$, while \eqref{eq:fluxr-}--\eqref{eq:fluxs-} are defined at $x_1, \cdots, x_N$.
The stabilization functions $\tau_{pu}^+$, $\tau_{pq}^+$, $\tau_{ru}^\pm$, $\tau_{rq}^\pm$, $\tau_{rp}^-$, $\tau_{su}^\pm$ $\tau_{sq}^\pm$, $\tau_{sp}^-$ are defined on $\partial\calT_h$, and are usually piece-wise constant. Due to the nonlinearity of $F(u)$, $\tau_\calF(\cdot, \cdot)$ can be a nonlinear function of $\widehat{u}_h$ and $u_h$, and is taken to be $0$ when $F \equiv 0$. These functions, when satisfing certain conditions to be specified later, ensure that the above problem has a unique solution.

It remains to impose the transmission conditions which allows us to solve for the globally coupled unknowns $\{\widehat{u}_{hi}\}_{i=0}^N$, $\{\widehat{q}_{hi}\}_{i=0}^N$, and $\{\widehat{p}^{\;-}_{hi}\}_{i=1}^N$:
\begin{equation}
\label{eq:transmission}
\jmp{\widehat{p}_h } (x_i) = 0, \,\, \jmp{\widehat{r}_h } (x_i) = 0, \,\, \mbox{ and }\,
\jmp{\widehat{\calF}_h } (x_i) = 0
\quad \mbox{ for all } i =1,\dots,N,
\end{equation}
where$\jmp{\omega}(x_i):=\omega(x_i^-)-\omega(x_i^+)$ for any function $\omega$. Here, to honor the periodic boundary conditions, we define $\jmp{\widehat{\omega}_h } (x_N) = \widehat{\omega}_h(x_N^-) - \widehat{\omega}_h(x_0^+)$ for $\omega \in \{p, r, \calF\}$. 
This completes the definition of the HDG method.

\subsubsection{Time discretization}\label{sec:time_discretization}

For time discretization of the KdV equation, we employ implicit time-marching schemes to \rev{mitigate the severe time-step restriction} due to the fifth-order spatial derivative. One may use  implicit time-marching schemes such as BDF or DIRK methods. Here for simplicity, we apply the following second-order midpoint rule \cite{BonaChenKarakashianXing13,ChenCockburnDongDG16,Dong17} to discretize the time derivative.

Let $0 = t_0 < t_1 < \cdots < t_N = T$ be a partition of the interval $[0, T]$
and $\Delta t_j = t_{j+1} -t_j$. At the initial time $t=0$, we choose the initial approximation 
to be the HDG approximate solution of the stationary fifth-order equation
\begin{equation}\label{eq:stationary_eqn}
\rev{\mathcal{D}(u)} +\gamma u =\tilde{f},
\end{equation}
where $\mathcal{D}(u):= \alpha u_{xxx} + \beta u_{xxxxx} + F(u)_x $, $\gamma=1$ and $\tilde{f}= \mathcal{D}(u_0) +\gamma u_0$.
At later time $t=t_{j+1}$ for $j=0, \cdots, N-1$, we let the approximation $u^{j+1}$ to $u(\cdot, t_{j+1})$ be $$u^{j+1}=2 u^{j,1}- u^j,$$ where $u^{j,1}$ is the solution of the equation
$$\frac{u^{j,1}-u^j}{\frac{1}{2}\Delta t_j}+\mathcal{D}(u^{j,1})=f.$$

It is easy to see that the equation above for $u^{j,1}$ can be rewritten into the form of \eqref{eq:stationary_eqn} with $\gamma=2/\Delta t_j$ and $\tilde{f}=f+\gamma u^j$. Therefore, at each time step, we use the HDG method to solve the stationary fifth-order equation \eqref{eq:stationary_eqn}.
To do that, we rewrite \eqref{eq:stationary_eqn} into the following first-order system:
\begin{equation}
\label{eq:equationst}
q - u_x \,=\, 0, \quad p - q_x\,=\, 0,\quad   r - p_x \,=\, 0, \quad s - r_x \,=\, 0, \quad \gamma u+ \calF(s,p,u)_x \,=\,\tilde{f}.
\end{equation}
In order to write the HDG formulation in a compact form, we use the notation $$(\varphi, v):=\sum_{i=1}^N (\phi, v)_{I_i},  \quad \langle \varphi, v n\rangle:=\sum_{i=1}^N \langle \varphi, v n\rangle_{\partial I_i}.$$
The HDG approximations $(u_h, q_h, p_h, r_h, s_h, \widehat{u}_h, \widehat{q}_h, \widehat{p}_h^{\;-})\in \left[W_h^{k}\right]^5\times \left[ L^2(\Eh) \right]^3$ for \eqref{eq:equationst}
satisfy 
\begin{equation}\label{eq:schemest}
\begin{split}
({q}_h,{v}) + (u_h,v_x) - \langle \widehat{u}_h,{v}n \rangle & = 0, \\
({p}_h,{z}) + (q_h,z_x) - \langle \widehat{q}_h,{z}n \rangle & = 0,\\
({r}_h,{w})  + (p_h,w_x  - \langle \widehat{p}_h,{w}n \rangle  & = 0, \\
({s}_h,{\phi})  + (r_h,\phi_x)  - \langle \widehat{r}_h,{\phi}n \rangle  & = 0,\\
(\gamma u_h,{\psi})  - (\calF(s_h,p_h,u_h), \psi_x)  + \langle \widehat{\calF}_h, {\psi}n \rangle  & = (f, \psi),
\end{split}
\end{equation}
for all $(v, z, w, \phi, \psi)\,\in\,\left[W_h^{k}\right]^5$. Here the numerical traces $(\widehat{p}_h^{\;+}, \widehat{r}_h, \widehat{s}_h, \widehat{\calF}_h)$ are defined in the same way as \eqref{eq:fluxes}. The globally coupled degrees of freedom are those associated with  $\{\widehat{u}_{hi}\}_{i=0}^N$, $\{\widehat{q}_{hi}\}_{i=0}^N$, and $\{\widehat{p}^{\;-}_{hi}\}_{i=1}^N$ and they are determined by the transmission conditions \eqref{eq:transmission}. Note that the initial approximations at $t=0$ are obtained by using \eqref{eq:schemest}, so they satisfy the equations \eqref{eq:scheme1}--\eqref{eq:scheme4} for the time-dependent problem.

\subsection{Stability of the semi-discrete scheme}
\begin{theorem}
\label{thm:stability}
The semi-discrete scheme \eqref{eq:scheme} for the fifth-order KdV equation is $L^2$-stable if the stabilization functions satisfy the following conditions 
\[
\begin{cases}
\tau_\calF \ge \tilde{\tau}:= \frac{1}{(\widehat{u}_h - u_h)^2} \int_{\widehat{u}_h}^{u_h} \left( F(\widehat{u}_h) - F(s) \right) \cdot n ds,\\
\tau_{su}^+ \ge -\frac{\alpha}{\beta} \tau_{pu}^+ + \frac{1}{2} {\tau_{pu}^+}^2 \quad \mbox{ or } \quad \tau_{rq}^+ \le \frac{\alpha}{2 \beta}  - \frac{1}{2} {\tau_{pq}^+}^2,\\
\left(\tau_{su}^+ + \frac{\alpha}{\beta} \tau_{pu}^+ - \frac{1}{2} {\tau_{pu}^+}^2\right) \left(-\tau_{rq}^+ + \frac{\alpha}{2 \beta}  - \frac{1}{2} {\tau_{pq}^+}^2\right) \ge \frac{1}{4} \left( \tau_{sq}^+ - \tau_{ru}^+ + \frac{\alpha}{\beta} \tau_{pq}^+ - \tau_{pu}^+\tau_{pq}^+ \right)^2,\\
\tau_{su}^- \ge \frac{1}{2} (\tau_{sp}^- + \frac{\alpha}{\beta})^2,\\
\tau_{rq}^- \le -\frac{\alpha}{2 \beta} - \frac{1}{2} {\tau_{rp}^-}^2,\\
-\tau_{su}^- \left( \tau_{rq}^- + \frac{\alpha}{2 \beta} \right) \ge \frac{1}{4}(\tau_{sq}^- - \tau_{ru}^-)^2.
\end{cases}
\]
\end{theorem}

\begin{remark}
\rev{As complicated as these conditions  in Theorem \ref{thm:stability} may appear, it is nevertheless easy to identify} stabilization functions to satisfy them. For example, we can take
\begin{equation}
\label{eq:tau_cond}
\begin{split}
&\tau_{pu}^+ = \tau_{pq}^+ =   \tau_{sq}^+ = \tau_{ru}^+=0, \quad \tau_{su}^+ \ge 0, \quad \tau_{rq}^+ \le \frac{\alpha}{2 \beta} ,\\
&\tau_{sp}^- = \tau_{rp}^- =    \tau_{sq}^- = \tau_{ru}^-=0,  \quad  \tau_{su}^- \ge\frac{1}{2}( \frac{\alpha}{\beta})^2,\quad \tau_{rq}^- \le -\frac{\alpha}{2 \beta} 
\end{split}
\end{equation}
given that $\beta<0$. Then the numerical traces in \eqref{eq:fluxp+}-\eqref{eq:fluxs-} have the simple form
\begin{alignat*}{2}
\widehat{p}_h^{\;+} & = p_h^+  \qquad &\mbox{ at } x_{i-1}^+,\\
\widehat{r}_h & = r_h +  \tau_{rq} (\widehat{q}_h - q_h )\,n\qquad &\mbox{ at } x_{i-1}^+ \mbox{  and } x_{i}^-,\\
\widehat{s}_h & = s_h + \tau_{su}\,( {\widehat{u}_{h}} - u_h)\,n \qquad  &\mbox{ at } x_{i-1}^+ \mbox{  and } x_{i}^-.
\end{alignat*}
To have $\tau_\calF \ge \tilde{\tau}$, we take $\tau_\calF=0$ if  $F(u)\equiv 0$. If $F$ is nonzero, we can take
$$\tau_\calF\ge \frac{1}{2}\sup_{s\in J(u_h, \widehat{u}_h)}|F'(s)|,$$
where $J(u_h, \widehat{u}_h)=[\min\{u_h,\widehat{u}_h\}, \max\{u_h,\widehat{u}_h\}]$, because
$$\tilde{\tau}=\frac{1}{(u_h-\widehat{u}_h)^2}{\int_{\widehat{u}_h}^{u_h}F'(\xi)(s-\widehat{u}_h) n\, ds}\le \frac{1}{2} \sup_{s\in J(u_h, \widehat{u}_h)}|F'(s)|.$$
There are also other choices of $\tau_\calF$ that satisfy the condition  $\tau_\calF\ge \tilde{\tau}$; see \cite{NguyenPeraireCockburn09}.
\end{remark}

Now let us prove the stability result in Theorem \ref{thm:stability}, and the proof is similar to that in \cite{Dong17, NguyenPeraireCockburn09}.
\begin{proof}
Taking $\psi = u_h$, $v = -(\alpha p_h + \beta s_h)$, $\phi = \beta q_h$, $z = \alpha q_h + \beta r_h$, and $w = -\beta p_h$ in \eqref{eq:scheme1} -- \eqref{eq:scheme5}, adding the five resulting equations together, and performing trivial algebraic manipulations, we have
\begin{alignat*}{1}
(f, u_h) = & (u_{ht}, u_h) - (F(u_h), u_{hx}) + \langle\widehat{\calF}_h, u_h \cdot n \rangle\\
& -(\alpha p_h + \beta s_h, u_{hx}) - (u_h, (\alpha p_h + \beta s_h)_x) + \langle \widehat{u}_h, (\alpha p_h + \beta s_h) n\rangle\\
& + (r_h, \beta q_{hx}) - \langle \widehat{r}_h, \beta q_h \cdot n \rangle - (p_h, \beta p_{hx}) + \langle \widehat{p}_h, \beta p_h \cdot n\rangle\\
& + (q_h, (\alpha q_h + \beta r_h)_x) - \langle \widehat{q}_h, (\alpha q_h + \beta r_h) \cdot n\rangle.
\end{alignat*}
Using integration by parts, the fact that $\widehat{u}_h$, $\widehat{q}_h$, $\widehat{p}_h$, $\widehat{r}_h$, $\widehat{\calF}_h$ are single-valued, and the periodic boundary condition, 
we obtain 
\begin{alignat*}{1}
(f, u_h)  =& \frac{1}{2}\frac{d}{dt} \int_{\mathcal{T}_h} u_h^2 dx + \Phi_1 + \Phi_2 + \Phi_3,\\
\intertext{ where }
\Phi_1 =  & -\langle \alpha(\widehat{p}_h - p_h) + \beta (\widehat{s}_h - s_h), (\widehat{u}_h - u_h) \cdot n \rangle + \langle \widehat{r}_h - r_h, \beta (\widehat{q}_h - q_h) \cdot n \rangle\\
& + \frac{\alpha}{2} \langle (\widehat{q}_h - q_h)^2, n \rangle - \frac{\beta}{2} \langle (\widehat{p}_h - p_h)^2, n \rangle, \\
\Phi_2  = & -(F(u_h), u_{hx}) - \langle F(\widehat{u}_h) - \tau_\calF (\widehat{u}_h - u_h) \cdot n, (\widehat{u}_h - u_h) \cdot n \rangle,\\
\Phi_3  = & \frac{\beta}{2} \langle \widehat{p}_h^{\,2}, n \rangle.
\end{alignat*}
We immediately realize that $\Phi_3 = 0$ due to the periodicity. Next, we plug in the definition of the numerical fluxes \eqref{eq:fluxes} to simplify $\Phi_1$ and $\Phi_2$. Indeed, $\Phi_1$ can be rewritten as $\Phi_1 = \Phi_1^+ + \Phi_1^- $, where
\begin{alignat*}{1}
\Phi_1^+  =& \langle a_{uu}, (\widehat{u}_h - u_h)^2 \rangle_{\Eh^+} + \langle a_{qq}, (\widehat{q}_h - q_h)^2\rangle_{\Eh^+} + \langle a_{uq}, (\widehat{u}_h - u_h)(\widehat{q}_h - q_h)\rangle_{\Eh^+}\\
\intertext{ and }
\Phi_1^-  =& \langle b_{uu}, (\widehat{u}_h - u_h)^2 \rangle_{\Eh^-} + \langle b_{qq}, (\widehat{q}_h - q_h)^2\rangle_{\Eh^-} + \langle b_{pp}, (\widehat{p}_h - p_h)^2\rangle_{\Eh^-}\\
&+ \langle b_{uq}, (\widehat{u}_h - u_h)(\widehat{q}_h - q_h)\rangle_{\Eh^-} + \langle b_{up}, (\widehat{u}_h - u_h)(\widehat{p}_h - p_h)\rangle_{\Eh^-}\\
& + \langle b_{qp}, (\widehat{p}_h - p_h)(\widehat{q}_h - q_h)\rangle_{\Eh^-}.
\end{alignat*}
Here
\begin{alignat*}{1}
a_{uu} = & -\beta \tau_{su}^+ - \alpha \tau_{pu}^+ + \frac{\beta}{2} {\tau_{pu}^+}^2,\qquad
a_{qq} =   \beta \tau_{rq}^+ - \frac{\alpha}{2} + \frac{\beta}{2}{\tau_{pq}^+}^2,\\
a_{uq} = & -\beta \tau_{sq}^+ + \beta \tau_{ru}^+ - \alpha \tau_{pq}^+ + \beta \tau_{pu}^+ \tau_{pq}^+,\\
b_{uu} = & -\beta \tau_{su}^-,\qquad
b_{qq} =   \beta \tau_{rq}^- + \frac{\alpha}{2}, \qquad
b_{pp} =  -\frac{\beta}{2},\\
b_{uq} = & -\beta \tau_{sq}^- + \beta \tau_{ru}^-,\qquad
b_{up} =  -\beta \tau_{sp}^- - \alpha, \qquad
b_{qp} =  \beta \tau_{rp}^-.
\end{alignat*}

To simplify $\Phi_2$, we introduce $G(s)$, the antiderivative of $f(s)$. We have
\begin{alignat*}{1}
\Phi_2 & = -\langle G(u_h), n \rangle + \langle F(\widehat{u}_h), (\widehat{u}_h - u_h) \cdot n \rangle + \langle \tau_\calF, (\widehat{u}_h - u_h)^2\rangle\\
& = -\langle \int_{\widehat{u}_h}^{u_h} F(s) ds, n\rangle + \langle \int_{\widehat{u}_h}^{u_h} F(\widehat{u}_h)ds, n\rangle + \langle \tau_\calF, (\widehat{u}_h - u_h)^2\rangle\\
& = \langle \tau_\calF - \tilde{\tau}, (\widehat{u}_h - u_h)^2 \rangle,
 \intertext{where } \tilde{\tau} &= \frac{1}{(\widehat{u}_h - u_h)^2} \int_{\widehat{u}_h}^{u_h} \left( F(\widehat{u}_h) - F(s) \right) \cdot n ds.
\end{alignat*}
A sufficient condition for the $L^2$-stability is then $\Phi_2 \ge 0$, $\Phi_1^+ \ge 0$, and $\Phi_1^- \ge 0$. That is to ask
\[
\begin{cases}
& \tau_\calF \ge \tilde{\tau},\\
& a_{uu} \ge 0, \quad a_{qq} \ge 0, \quad 4 a_{uu} a_{qq} \ge a_{uq}^2,\\
& b_{uu} \ge 0, \quad b_{qq} \ge 0, \quad b_{pp} \ge 0, \quad 4 b_{uu} b_{qq} \ge b_{uq}^2,  \quad 4 b_{uu} b_{pp} \ge b_{up}^2, \quad 4 b_{pp} a_{qq} \ge b_{pq}^2.
\end{cases}
\]
Plugging in the definitions of $a_{*}$'s and $b_*$'s above, we have
\[
\begin{cases}
&-\beta \tau_{su}^+ - \alpha \tau_{pu}^+ + \frac{\beta}{2} {\tau_{pu}^+}^2 \ge 0,\\
&\beta \tau_{rq}^+ - \frac{\alpha}{2} + \frac{\beta}{2} {\tau_{pq}^+}^2 \ge 0,\\
&(-\beta \tau_{su}^+ - \alpha \tau_{pu}^+ + \frac{\beta}{2} {\tau_{pu}^+}^2)(\beta \tau_{rq}^+ - \frac{\alpha}{2} + \frac{\beta}{2} {\tau_{pq}^+}^2) \ge \frac{1}{4} \left( -\beta \tau_{sq}^+ + \beta \tau_{ru}^+ - \alpha \tau_{pq}^+ + \beta \tau_{pu}^+ \tau_{pq}^+\right)^2,
\end{cases}
\]
and
\[
\begin{cases}
&\tau_{su}^- \ge 0, \quad \tau_{rq}^- \le -\frac{\alpha}{2 \beta},\\
& -4 \beta \tau_{su}^- (\beta \tau_{rq}^- + \frac{\alpha}{2}) \ge \left(\beta \tau_{sq}^- + \beta \tau_{ru}^- \right)^2,\\
& 2 \beta^2 \tau_{su}^- \ge \left( \alpha + \beta \tau_{sp}^- \right)^2,\\
& -2 \beta (\frac{\alpha}{2} + \beta \tau_{rq}^-) \ge \beta^2 {\tau_{rp}^-}^2.
\end{cases}
\]
It is easy to see that these conditions are equivalent to those listed in the theorem. We thus complete the proof.
\end{proof}

\begin{remark}
With little to no change to the proof, we can show that the scheme is also $L^2$-stable for the Dirichlet-type boundary condition such as the following
\begin{equation}
\label{eq:BC_D}
u|_{0,L} = 0, \quad q|_{0,L} = 0, \quad p|_L = 0. 
\end{equation}
\end{remark}

\section{A Priori Error Analysis}

\label{sec:erroranalysis}

\rev{The goal of this section is to establish the optimal} accuracy of our scheme for the linear case, that is, \rev{when} we assume $F \equiv 0$. Toward that end, we define the errors of the approximations $\omega_h$ and of the numerical fluxes $\widehat{\omega}_h$ as
\[
e_\omega = \omega - \omega_h, \quad \widehat{e}_\omega = \omega - \widehat{\omega}_h \quad \mbox{ for } \omega \in \{u, q, p, r, s\}.
\]
The analysis is projection-based, see also \cite{CockburnGopalakrishnanSayas10, ChenCockburnDongHDG16, Dong17}. We first define a projection
\[
\varPi: \left[ L^2(\calT_h) \right]^5 \longrightarrow \left[ W_h^k \right]^5
\]
that is inspired by the scheme, in particular the definition of the numerical fluxes \eqref{eq:fluxes}. Then we prove the optimal approximation property of this projection error
\[
\delta_\omega = \omega - \varPi \omega \quad \mbox{ for } \omega \in \{u, q, p, r, s\},
\]
and that the projections of the error, defined as
\[
\varepsilon_\omega = \varPi \omega - \omega_h \quad \mbox{ for } \omega \in \{u, q, p, r, s\},
\]
also converge optimally. We can then conclude that the errors $e_\omega$ converge optimally due to that $e_\omega = \delta_\omega + \varepsilon_\omega$ and the triangle inequality.

Let us first define the projection $\varPi: \left[ L^2(\calT_h) \right]^5 \longrightarrow \left[ W_h^k \right]^5$. 
For $(u, q, p, r, s) \in \left[ L^2(\calT_h) \right]^5$, we find in $\left[ W_h^k \right]^5$ the element $(\varPi u, \varPi q, \varPi p, \varPi r, \varPi s)$ whose restriction on element $I_i=(x_{i-1}, x_i)$ satisfies the following condition
\begin{equation} \label{eq:proj}
\begin{cases}
(\delta u, v)_{I_i} = 0, \;\; (\delta q, z)_{I_i} = 0, \;\; (\delta p, w)_{I_i} = 0, \;\; & (\delta r, \phi)_{I_i} = 0, \;\; (\delta s, \psi)_{I_i} = 0,\\
\delta_p - \tau_{pu}^+ \delta_u n - \tau_{pq}^+ \delta_q n = 0 & \mbox{ at } x_{i-1}^+\\
\delta_r - \tau_{ru}^+ \delta_u n - \tau_{rq}^+ \delta_q n = 0 & \mbox{ at } x_{i-1}^+\\
\delta_s - \tau_{su}^+ \delta_u n - \tau_{sq}^+ \delta_q n = 0 & \mbox{ at } x_{i-1}^+\\
\delta_r - \tau_{ru}^- \delta_u n - \tau_{rq}^- \delta_q n - \tau_{rp}^- \delta_p n= 0 & \mbox{ at } x_{i}^-\\
\delta_s - \tau_{su}^- \delta_u n - \tau_{sq}^- \delta_q n - \tau_{sp}^- \delta_p n= 0 & \mbox{ at } x_{i}^-
\end{cases}
\end{equation}
for any $(v, z, w, \phi, \psi) \in \left[P_{k-1}(K)\right]^5,$ where we have used the notation $\delta_\omega:=\omega-\varPi\omega$ for any $\omega\in L^2(\calT_h)$.

To state our error estimate, we use the following notation. The $H^s(D)$-norm is denoted by $\|\cdot\|_{s, D}$.
We drop the first subindex if $s=0$,  and the second one if $D=\Omega$ or $D={\mathcal T}_h$.
We are now ready to state our error estimate which is given as an upper bound for the following quantity
\begin{equation}
\label{eq:error_def}
\lVert \varepsilon \rVert^2:= \lVert \varepsilon_u \rVert^2 + \lVert \varepsilon_q \rVert^2 + \lVert \varepsilon_p \rVert^2 + \lVert \varepsilon_r \rVert^2 + \lVert \varepsilon_s \rVert^2 + \lVert \varepsilon_{u_t} \rVert^2.
\end{equation}

\begin{theorem}
\label{thm:error}
Suppose that $u$ is the solution to the fifth-order Korteweg-de Vries equation \eqref{eq:prob} with $F \equiv 0$, $\alpha = 0$ and $\beta=-1$, and $u_h$ is the approximate solution given by the HDG scheme \eqref{eq:scheme}--\eqref{eq:transmission} with the stabilization functions $\tau_{su}^\pm=1, \tau_{rq}^\pm=-1$ and all others being zero. 
Then for $k > 0$, the projection of the error \eqref{eq:error_def} satisfies
\[
\lVert \varepsilon(t) \rVert  \le C h^{k + 1}.
\]
\end{theorem} 

To prove Theorem \ref{thm:error}, we need the following three lemmas. They are the energy equalities, the approximation property of the projection, and the requirement on the initial condition.

\begin{lemma}[Energy identities]
\label{lemma_energy}
We have that
\begin{subequations}
\label{eq:energyidentities}
\begin{alignat}{1}
\label{eq:energyid1}
0 = & \frac{1}{2}\frac{d}{dt} \lVert \varepsilon_u \rVert^2 + S_1 + T_1,\\
\label{eq:energyid2}
0 = & \frac{1}{2}\frac{d}{dt} \lVert \varepsilon_q \rVert^2 + S_2 + T_2,\\
\label{eq:energyid3}
0 = & \frac{1}{2}\frac{d}{dt} \lVert \varepsilon_p \rVert^2 + S_3 + T_3,\\
\label{eq:energyid4}
0 = & \frac{1}{2}\frac{d}{dt} \lVert \varepsilon_r \rVert^2 + S_4 + T_4,\\
\label{eq:energyid5}
0 = & \frac{1}{2}\frac{d}{dt} \lVert \varepsilon_s \rVert^2 + S_5 + T_5,\\
\label{eq:energyid6}
0 = & \frac{1}{2}\frac{d}{dt} \lVert \varepsilon_{u_t} \rVert^2 + S_6 + T_6.
\end{alignat}
\end{subequations}
Moreover, $$T :={\displaystyle \sum_{k=1}^6} T_k \ge 0.$$
Here
\begin{alignat*}{1}
S_1 = & (\delta_{ut}, \varepsilon_u) - (\delta_q, \alpha \varepsilon_p + \beta \varepsilon_s) + (\delta_p, \alpha \varepsilon_q + \beta \varepsilon_r) - (\delta_r, \beta \varepsilon_p) + (\delta_s, \beta \varepsilon_q)\nonumber\\
T_1 = & -\langle \alpha (\widehat{\varepsilon}_p - \varepsilon_p) + \beta (\widehat{\varepsilon}_s - \varepsilon_s), (\widehat{\varepsilon}_u - \varepsilon_u) \cdot n \rangle + \beta \langle (\widehat{\varepsilon}_r - \varepsilon_r) , (\widehat{\varepsilon}_q - \varepsilon_q) \cdot n \rangle \nonumber\\
& + \frac{\alpha}{2} \langle (\widehat{\varepsilon}_q - \varepsilon_q) ^2, n\rangle - \frac{\beta}{2} \langle (\widehat{\varepsilon}_p - \varepsilon_p)^2, n \rangle \nonumber\\
S_2 = & (\delta_{qt}, \varepsilon_q) + (\delta_p, \varepsilon_{ut}) + (\varepsilon_p, \varepsilon_{ut})\nonumber \\
T_2 = & \langle (\widehat{\varepsilon}_q - \varepsilon_q) , (\widehat{\varepsilon}_{ut} - \varepsilon_{ut})  \cdot n\rangle\nonumber\\
S_3 = & (\delta_{pt}, \varepsilon_{p}) + (\delta_{r}, \varepsilon_{qt}) - (\delta_{qt}, \varepsilon_{r}) - (\delta_{s}, \varepsilon_{ut}) - (\varepsilon_{s}, \varepsilon_{ut})\nonumber\\
T_3 = & \langle (\widehat{\varepsilon}_p - \varepsilon_p) , (\widehat{\varepsilon}_{qt} - \varepsilon_{qt}) \cdot n \rangle - \langle (\widehat{\varepsilon}_r - \varepsilon_r) , (\widehat{\varepsilon}_{ut} - \varepsilon_{ut}) \cdot n \rangle\nonumber\\
S_4 = & (\delta_{rt}, \varepsilon_{r}) + (\delta_{s}, \varepsilon_{pt}) - (\delta_{pt}, \varepsilon_{s} + \frac{\alpha}{\beta}\varepsilon_p) + \frac{1}{\beta}(\delta_{ut}, \varepsilon_{qt}) - \frac{1}{\beta}(\delta_{qt}, \varepsilon_{ut})\nonumber\\
T_4 = & \langle (\widehat{\varepsilon}_r - \varepsilon_r) , (\widehat{\varepsilon}_{pt} - \varepsilon_{pt}) \cdot n \rangle
- \langle (\widehat{\varepsilon}_{qt} - \varepsilon_{qt}) , (\widehat{\varepsilon}_s - \varepsilon_s) \cdot n \rangle\nonumber\\
& - \frac{1}{2 \beta} \langle (\widehat{\varepsilon}_{ut} - \varepsilon_{ut})^2, n\rangle
 - \frac{\alpha}{\beta} \langle (\widehat{\varepsilon}_p - \varepsilon_p) , (\widehat{\varepsilon}_{qt} - \varepsilon_{qt})  \cdot n\rangle \nonumber \\
 S_5 = & (\delta_{st}, \varepsilon_s) - \frac{1}{\beta}(\delta_{ut}, \varepsilon_{rt}) + \frac{1}{\beta}(\delta_{rt}, \varepsilon_{ut}) + \frac{1}{\beta}(\delta_{qt}, \varepsilon_{pt}) - \frac{1}{\beta}(\delta_{pt}, \varepsilon_{qt}) - \frac{\alpha}{\beta}(\delta_{r}, \varepsilon_{rt})\nonumber\\
 T_5 = & \langle (\widehat{\varepsilon}_{rt} - \varepsilon_{rt}) , (\widehat{\varepsilon}_s - \varepsilon_s) \cdot n\rangle
 + \frac{1}{\beta} \langle (\widehat{\varepsilon}_{pt} - \varepsilon_{pt}) , (\widehat{\varepsilon}_{ut} - \varepsilon_{ut}) \cdot n \rangle -\frac{1}{2 \beta} \langle (\widehat{\varepsilon}_{qt} - \varepsilon_{qt})^2, n \rangle \nonumber \\
 S_6 = & (\delta_{u_{tt}}, \varepsilon_{ut}) - (\delta_{qt}, \alpha \varepsilon_{pt} + \beta \varepsilon_{st})
 + (\delta_{pt}, \alpha \varepsilon_{qt} + \beta \varepsilon_{rt}) + (\delta_{st}, \beta \varepsilon_{qt}) - (\delta_{rt}, \beta \varepsilon_{pt})
 \nonumber\\
 T_6 = & -\langle \beta (\widehat{\varepsilon}_{st} - \varepsilon_{st}) + \alpha (\widehat{\varepsilon}_{pt} - \varepsilon_{pt}) , (\widehat{\varepsilon}_{ut} - \varepsilon_{ut}) \cdot n \rangle +\frac{\alpha}{2} \langle (\widehat{\varepsilon}_{qt} - \varepsilon_{qt})^2, n\rangle \nonumber\\
& + \beta \langle (\widehat{\varepsilon}_{rt} - \varepsilon_{rt}) , (\widehat{\varepsilon}_{qt} - \varepsilon_{qt}) \cdot n \rangle  - \frac{\beta}{2} \langle (\widehat{\varepsilon}_{pt} - \varepsilon_{pt})^2, n \rangle. \nonumber
\end{alignat*}
\end{lemma}

The proof of this lemma is long and technical, thus given in Appendix \ref{sec:appendixA}.

The next lemma states that the projection $\varPi$ is optimal, that is, the projection error is of order $k+1$. To describe it, we define the auxiliary penalization functions
\[
\theta^\alpha_\omega:= \tau_{\alpha \omega}^- + \tau_{\alpha \omega}^+ - \tau_{\alpha p}^-\tau_{p \omega}^+ \mbox{ for } \alpha \in \{s, r\} \mbox{ and } \omega \in \{u, q\}.
\]

\begin{lemma} [Approximation properties of the projection.]
\label{lemma_approx}
When the penalization functions satisfy $\theta_q^s \theta_u^r - \theta_u^s \theta_q^r \neq 0$, we have that, for any $(u, q, p, r, s) \in \left[ H^{k+1}(K)\right]^5$,
\[
\lVert \delta_u \rVert + \lVert \delta_q \rVert + \lVert \delta_p \rVert + \lVert \delta_r \rVert + \lVert \delta_s \rVert \le C h^{k+1}.
\]
\end{lemma}
The proof of this lemma is given in Appendix \ref{sec:appendixB}.

The last lemma is about the optimal convergence of the initial approximations. It is usual to require that the projection of the error at the initial step $\epsilon(0)$ satisfies certain condition for the solution at later time to have optimal convergence, see also \cite{XuShu12, Dong17}.

\begin{lemma}[On the initial approximations]
\label{lemma_init}
Under the assumption in Theorem \ref{thm:error}, we have that
\begin{alignat*}{1}
&\lVert \varepsilon_\omega(0) \rVert \le C h^{k+2}, \quad ||\widehat{e}_\omega(0)||_{\infty, \Eh}\le Ch^{2k+1}\quad\mbox{ for }\omega=u, q, p, r, s \\
\intertext{and}
& \lVert \varepsilon_{ut}(0) \rVert \le C h^{k+1}.
\end{alignat*}
\end{lemma}
The proof of this Lemma involves the  error analysis of the HDG method for the stationary fifth-order equation. It is lengthy, thus we give a detailed sketch in Appendix \ref{sec:appendixC}, and proceed here to prove the main theorem in the remaining of this section.

\noindent {\bf Proof of Theorem \ref{thm:error}:} For simplicity, we assume $\alpha = 0$.  We first add the six energy identities \eqref{eq:energyidentities} to obtain
\[
\frac{d}{dt} \lVert \varepsilon(t) \rVert^2 + S + T = 0 \quad \mbox{ with } S = \sum_{i=1}^6 S_i, \quad T = \sum_{i=1}^6 T_i.
\]
Next, we integrate it from $0$ to $t$ and apply the result $T \ge 0$ from Lemma \ref{lemma_energy}. We have
\[
\lVert \varepsilon(t) \rVert^2 \le  \lVert \varepsilon(0) \rVert^2 + \left| \int_0^t S(s) ds\right|.
\]

To estimate $\int_0^t S(s) ds$, we rewrite $S$ as $S = I_1 + I_2 + I_3$ with

\begin{alignat*}{1}
I_1 = & (\delta_{ut}, \varepsilon_u) - (\delta_q, \beta \varepsilon_s) + (\delta_p, \beta \varepsilon_r) - (\delta_r, \beta \varepsilon_p) + (\delta_s, \beta \varepsilon_q)\\
& +(\delta_{qt}, \varepsilon_q) + (\delta_p, \varepsilon_{ut}) +(\delta_{pt}, \varepsilon_{p})  - (\delta_{qt}, \varepsilon_{r}) - - (\delta_{s}, \varepsilon_{ut})\\
&+(\delta_{rt}, \varepsilon_{r}) - (\delta_{pt}, \varepsilon_{s})  - \frac{1}{\beta}(\delta_{qt}, \varepsilon_{ut})
+(\delta_{st}, \varepsilon_s) + \frac{1}{\beta}(\delta_{rt}, \varepsilon_{ut}) + (\delta_{u_{tt}}, \varepsilon_{ut}),\\
I_2 = & (\delta_{r}, \varepsilon_{qt}) + (\delta_{s}, \varepsilon_{pt}) + \frac{1}{\beta}(\delta_{ut}, \varepsilon_{qt})
- \frac{1}{\beta}(\delta_{ut}, \varepsilon_{rt})  + \frac{1}{\beta}(\delta_{qt}, \varepsilon_{pt}) - \frac{1}{\beta}(\delta_{pt}, \varepsilon_{qt})\\
&  - (\delta_{qt}, \beta \varepsilon_{st}) + (\delta_{pt}, \beta \varepsilon_{rt}) + (\delta_{st}, \beta \varepsilon_{qt}) - (\delta_{rt}, \beta \varepsilon_{pt}),\\
I_3 = & (\varepsilon_p, \varepsilon_{ut})+(\varepsilon_s, \varepsilon_{ut}).
\end{alignat*}
Using the Cauchy inequality and the approximation property of the projection Lemma \ref{lemma_approx}, we know
$|I_1| \le C h^{k+1} \lVert \varepsilon \rVert$ which means that
\[
\int_0^t \left | I_1 \right | dt \le C h^{k+1} \int_0^t \lVert \varepsilon(s) \rVert ds \le \frac{C}{\epsilon} T h^{2k+2} + C \epsilon \int_0^t \lVert \varepsilon \rVert^2 dt.
\]
For $I_2$, we have
\begin{alignat*}{1}
\int_0^t I_2 dt \le & \left[ (\delta_r, \varepsilon_q) + (\delta_s, \varepsilon_p) + \frac{1}{\beta}(\delta_{ut}, \varepsilon_q) - \frac{1}{\beta} (\delta_{ut}, \varepsilon_r) + \frac{1}{\beta} (\delta_{qt}, \varepsilon_p) - \frac{1}{\beta} (\delta_{pt}, \varepsilon_q) \right]_0^t\\
&+ \int_0^t \left[ - (\delta_{rt}, \varepsilon_q) - (\delta_{st}, \varepsilon_p) - \frac{1}{\beta}(\delta_{utt}, \varepsilon_q) + \frac{1}{\beta} (\delta_{utt}, \varepsilon_r) - \frac{1}{\beta} (\delta_{qtt}, \varepsilon_p) + \frac{1}{\beta} (\delta_{ptt}, \varepsilon_q) \right] dt,
\end{alignat*}
which means
\[
\left| \int_0^t I_2 dt\right| \le C h^{k+1} \left( \lVert \varepsilon(0) + \varepsilon(t) \rVert\right) + C h^{k+1} \int_0^t \lVert \varepsilon(s) \rVert ds.
\]
As to $I_3$, we have $|I_3| \le (\lVert \varepsilon_p \rVert+\lVert \varepsilon_s \rVert )  \lVert \varepsilon_{ut} \rVert \le \lVert \varepsilon \rVert^2$ which means that
\[
\int_0^t | I_3 | dt \le \int_0^t \lVert \varepsilon \rVert^2 dt.
\]
Putting these three inequalities together, we have
\begin{alignat*}{1}
\left | \int_0^t S dt\right | \le & \int_0^t |I_1| dt + \left| \int_0^t I_2 dt\right| + \int_0^t |I_3 | dt\\
\le & C (1 + t) h^{2k + 2} + C \lVert \varepsilon(0) \rVert^2 + \frac{1}{2} \lVert \varepsilon(t) \rVert^2 + C \int_0^t  \lVert \varepsilon(t) \rVert^2 dt.
\end{alignat*}
Hence, we have
\[
 \lVert \varepsilon(t) \rVert^2 \le C (1 + t) h^{2k + 2} + C \lVert \varepsilon(0) \rVert^2 + C \int_0^t  \lVert \varepsilon(t) \rVert^2 dt.
\]

As a final step, we apply $\lVert \varepsilon(0) \rVert \le C h^{k+1}$ from Lemma \ref{lemma_init} and the Gronwall's inequality to obtain $\lVert \varepsilon(t) \rVert^2 \le C h^{2 k + 2}$ which finishes the proof of Theorem \ref{thm:error}.

\section{Numerical Results}
\label{sec:numericaltest}

In this section, we test our scheme on four 
test problems to corroborate our theoretical results. The first two 
problems have the periodic boundary condition \eqref{eq:BC_P} while the last two 
are furnished with Dirichlet boundary condition \eqref{eq:BC_D}.
The spatial discretization polynomial degree $k$ is set to be between $0$ and $3$.
For time discretization, we use the second-order midpoint rule described in Section \ref{sec:time_discretization}.

\subsection{Test with periodic boundary condition}

For the first two 
tests, the components of the stabilization function  are taken to be $\tau_{rq}^\pm =-1, \tau_{su}^\pm=1$, $\tau_\calF=|F'(\widehat{u}_h)|$ and all other $\tau$ are zero, which satisfy \eqref{eq:tau_cond}.
We take $h = 2^{-n}$ for $n = 3, 4, 5, 6, 7$ for Problem 1 and Problem 2. The step size for time discretization is
$\Delta t = 0.1 \times h$ for $k = 0, 1$, $\Delta t = 0.1 \times h^2$ for $k = 2, 3$. This is just to ensure that the temporal errors are negligible comparing with the space error. We compute the orders of convergence of $p_{h}, q_{h}, u_{h}, r_{h}, s_{h}$ at the final time $T = 0.1$.

\subsubsection*{Test problems} The characteristics for the two 
problems are given in Table \ref{tab:problems_1to3}.

\begin{table}
\begin{center}
    \renewcommand{\tabcolsep}{0.4cm}
    \renewcommand{\arraystretch}{1.3}
\begin{tabular}{|c|c|c|}
\hline
 & Problem 1 & Problem 2 \\
 \hline
 Domain & $[0, 2\pi]$ & $[0, 2\pi]$\\
 \hline
 $(\alpha, \beta)$ & $(0, -1)$ & $(1, -1)$ \\
 \hline
 $F$ & 0 & $u + u^2 + u^3$ \\
 \hline
 $u(x,t)$ & $\sin(x+t)$ & $\sin(x+t)$ \\
 \hline
\end{tabular}
\end{center}
\caption{Setup for the first two problems.}
\label{tab:problems_1to3}
\end{table}

\subsubsection*{Results} We compute the orders of convergence for $p_{h}, q_{h}, u_{h}, r_{h}, s_{h}$ at the final time $T = 0.1$ for the linear equation in Problem 1 and the nonlinear equation in Problem 2. The results are listed in the Table \ref{tab:conv1} and Table \ref{tab:conv2} which clearly show optimal convergence.

\begin{table}
\begin{center}
    \renewcommand{\arraystretch}{1.3}
\begin{tabular}{|c|c|c|c|c|c|c|c|c|c|c|}
\hline
k & $e_{u}$ & order & $e_{q}$ & order & $e_{p}$ & order & $e_{r}$ & order & $e_{s}$ & order \\
\hline
\multirow{3}{*}{0}& 0.3789 & -    & 0.7027 & -     & 0.9919 & -    & 1.2084& - &1.4091 & - \\
                  & 0.2533 & 0.58 & 0.4138 & 0.76  & 0.6190 & 0.68 & 0.8222 &0.56& 1.0291& 0.45\\
                  & 0.1378 & 0.88 & 0.2182 & 0.92  & 0.3449 & 0.84 & 0.4805&0.78 & 0.6198& 0.73\\
                  & 7.04e-2& 0.97 & 0.1101 & 0.99  & 0.1798 & 0.94 & 0.2558&0.91 & 0.3339& 0.89\\
                  & 3.51e-2& 1.00 & 5.49e-2& 1.01  & 9.12e-2& 0.98 & 0.1310&0.97 & 0.1719& 0.96\\
\hline
\multirow{3}{*}{1}
&3.27e-2& -    & 5.01e-2 & -  &  5.21e-2& -    & 5.70e-2 & - &6.58e-2& - \\
&1.09e-2 &1.58 &1.30e-2 & 1.95& 1.32e-2 & 1.99 &1.35e-2  &2.07& 1.42e-2& 2.21\\
&3.16e-3 &1.78 &3.28e-3 & 1.99& 3.29e-3 & 2.00 &3.31e-3  &2.02& 3.36e-3& 2.08\\
&8.20e-4 &1.94 &8.21e-4 & 2.00& 8.22e-4 & 2.00 &8.24e-4  &2.01& 8.26e-4& 2.02\\
&2.05e-4 &2.00 &2.05e-4 & 2.00& 2.06e-4 & 2.00 &2.06e-4  &2.00& 2.06e-4& 2.01\\
\hline
\multirow{3}{*}{2}
&2.60e-3& -     & 3.12e-3 & -  & 3.20e-3 & -   & 3.18e-3& - &3.21e-3 & - \\
&3.76e-4 &2.79 &4.05e-4 & 2.95& 4.00e-4 & 3.00 &3.99e-4  &2.99& 4.30e-4& 2.90\\
&4.98e-5 &2.92  &4.98e-5 & 3.02& 5.00e-5 & 3.00&5.00e-5  &3.00& 5.02e-5& 3.10\\
&6.25e-6 &2.99  &6.24e-6 & 3.00& 6.25e-6 & 3.00&6.25e-6  &3.00& 6.23e-6& 3.01\\
&7.81e-7 &3.00  &7.81e-7 & 3.00& 7.81e-7 & 3.00&7.81e-7  &3.00& 7.78e-7& 3.00\\
\hline
\multirow{3}{*}{3}
&1.65e-4& -   & 1.76e-4 & -  & 1.77e-4 & -     & 1.76e-4& - &1.80e-4 & - \\
&1.20e-5 &3.79&1.21e-5 & 3.86& 1.21e-5 & 3.87  &1.21e-5 &3.86& 1.25e-5& 3.85\\
&8.06e-7 &3.89&8.05e-7 & 3.91& 8.06e-7 & 3.91  &8.06e-7  &3.91& 8.12e-7& 3.94\\
&5.04e-8 &4.00&5.03e-8 & 4.00& 5.04e-8 & 4.00  &5.04e-8  &4.00& 5.07e-8& 4.00\\
&3.15e-9 &4.00&3.15e-9 & 4.00& 3.15e-9 & 4.00  &3.15e-9  &4.00& 3.19e-9& 3.99\\
\hline
\end{tabular}
\end{center}
\caption{Errors and orders of convergence for problem 1.}
\label{tab:conv1}
\end{table}

\begin{table}
\begin{center}
    \renewcommand{\arraystretch}{1.3}
\begin{tabular}{|c|c|c|c|c|c|c|c|c|c|c|}
\hline
k & $e_{u}$ & order & $e_{q}$ & order & $e_{p}$ & order & $e_{r}$ & order & $e_{s}$ & order \\
\hline
\multirow{3}{*}{0}
&0.3876 & -     & 0.6285& -     & 0.9640 & -   & 1.2210& - &1.4680 & - \\
& 0.2597 &0.57  & 0.3547 & 0.83 & 0.5743 & 0.75& 0.7951 &0.62& 1.0260& 0.51\\
& 0.1397 & 0.89 & 0.1772 & 1.00 & 0.3012& 0.93 & 0.4405&0.85 & 0.590& 0.80\\
& 0.0710& 0.98  & 0.0865& 1.03  & 0.1513& 0.99 & 0.2277&0.95 & 0.3110& 0.92\\
& 3.55e-2& 1.00 & 4.23e-2& 1.03 & 7.53e-2& 1.01& 0.1153&0.98 & 0.1594& 0.96\\
\hline
\multirow{3}{*}{1}
&5.26e-2& -      & 5.56e-2 & -  &  5.19e-2& -    & 5.70e-2 & - &0.1971& - \\
& 4.58e-3 & 2.22 &1.35e-2 & 2.04& 1.31e-2 & 1.98 &1.31e-2  &2.01& 5.76e-2& 1.78\\
&1.07e-3 &2.10   &3.31e-3 & 2.03& 3.29e-3 & 2.00&3.28e-3  &2.00& 1.51e-2& 1.94\\
& 2.63e-4 &2.03  &8.23e-4 & 2.01& 8.22e-4 & 2.00&8.22e-4  &2.00& 3.81e-3& 1.98\\
& 6.55e-5 &2.01  &2.06e-4 & 2.00& 2.06e-4 & 2.00&2.06e-4  &2.00& 9.55e-4& 2.00\\
\hline
\multirow{3}{*}{2}
&1.44e-3 & -   & 3.31e-3 & -  & 3.28e-3 & -    & 3.46e-3& - &9.11e-3 & - \\
&1.79e-4 &3.01 &4.05e-4 & 2.96& 4.03e-4 & 3.02 &4.10e-4  &3.08& 1.22e-3& 2.91\\
&2.23e-5 &3.00 &4.98e-5 & 3.02& 5.01e-5 & 3.00 &5.03e-5  &3.03& 1.50e-4& 3.02\\
&2.79e-6 &3.00 &6.24e-6 & 3.00& 6.25e-6 & 3.00 &6.26e-6  &3.01& 1.88e-5& 2.99\\
&3.49e-7 &3.00 &7.81e-7 & 3.00& 7.81e-7 & 3.00 &7.82e-7  &3.00& 2.35e-6& 3.00\\
\hline
\multirow{3}{*}{3}
&1.25e-4 & -   & 2.21e-4 & -  & 3.14e-4 & -    & 5.53e-4& - &1.24e-4 & - \\
&9.68e-6 &3.69 &1.62e-5 & 3.77& 2.52e-5 & 3.64 &4.85e-5 &3.51& 1.14e-5& 3.44\\
&6.86e-7 &3.82 &1.11e-6 & 3.87& 1.73e-6 & 3.86 &3.22e-6  &3.91& 7.06e-6& 4.01\\
&4.28e-8 &4.00 &6.90e-8 & 4.00& 1.08e-7 & 4.00 &2.00e-7  &4.01& 4.35e-7& 4.02\\
&2.68e-9 &4.00 &4.31e-9 & 4.00& 6.75e-9 & 4.00 &1.26e-8  &3.99& 2.75e-8& 3.98\\
\hline
\end{tabular}
\end{center}
\caption{Errors and orders of convergence for problem 2.}
\label{tab:conv2}
\end{table}

\subsection{Tests with Dirichlet boundary condition}

Although our error analysis is for KdV equations with periodic boundary conditions, we would like to investigate the convergence rates of our HDG method for Dirichlet boundary conditions. In test problems 3 and 4, the stabilization functions are taken to be $\tau_{su}^\pm =\tau_{sq}^\pm =\tau_{ru}^\pm=1, \tau_{rq}^\pm=-1$, $\tau_\calF=|F'(\widehat{u}_h)|$ and others zero. 
The step size in time is again set to be $\Delta t = 0.1 \times h$ for $k = 0, 1$, and $\Delta t = 0.1 \times h^2$ for $k = 2, 3$.

\subsubsection*{Test problems} The characteristics for the two problems are given in Table \ref{tab:problems_4to6}.

\begin{table}
\begin{center}
    \renewcommand{\tabcolsep}{0.4cm}
    \renewcommand{\arraystretch}{1.3}
\begin{tabular}{|c|c|c|}
\hline
 & Problem 3 & Problem 4\\
 \hline
 Domain & $[0, \pi]$ & $[0, \pi]$ \\
 \hline
 $(\alpha, \beta)$ & $(0, -1)$ & $(1, -1)$\\
 \hline
 $F$ & 0 & $u + u^2 + u^3$\\
 \hline
 $u(x,t)$ & $t \cdot \sin(x)$ & $t \cdot \sin(x)$ \\
 \hline
\end{tabular}
\end{center}
\caption{Setup for the next two problems.}
\label{tab:problems_4to6}
\end{table}

\subsubsection*{Results}
We show the convergence orders in Tables \ref{tab:conv4} - \ref{tab:conv5} for the errors in all the five variables for the linear Problems 3 and the nonlinear problem 4. Optimal convergence rates are achieved for both the linear and the nonlinear cases for all the variables.

\begin{table}
\begin{center}
    \renewcommand{\arraystretch}{1.3}
\begin{tabular}{|c|c|c|c|c|c|c|c|c|c|c|}
\hline
k & $e_{u}$ & order & $e_{q}$ & order & $e_{p}$ & order & $e_{r}$ & order & $e_{s}$ & order \\
\hline
\multirow{3}{*}{0}
& 9.56e-3 & -   & 6.99e-3 & -    & 2.16e-2 & -    & 5.48e-2& - &9.03e-2 & - \\
& 7.20e-3 &0.41 & 4.38e-3 & 0.67 & 1.20e-2 & 0.85 & 3.17e-2 &0.79& 6.98e-2& 0.37\\
& 4.63e-3 &0.64 & 2.67e-3 & 0.71 & 6.27e-3 & 0.94 & 1.89e-2 &0.74& 4.95e-2& 0.50\\
& 2.67e-3 & 0.80& 1.51e-3 & 0.83 & 3.31e-3& 0.92   & 1.11e-2&0.76 & 3.00e-2& 0.72\\
& 1.44e-3& 0.89 & 8.02e-4 & 0.91 & 1.73e-3& 0.93  & 6.18e-3&0.85 & 1.65e-2& 0.86\\
\hline
\multirow{3}{*}{1}
& 4.49e-4 &-     &7.58e-4 & -   & 9.75e-4 & -   & 1.25e-3  &-&  1.92e-3& -\\
& 1.36e-4 &1.72  &1.76e-4 & 2.11& 2.39e-4 & 2.03& 3.21e-4  &1.97&  3.72e-4& 2.37\\
& 3.75e-5 &1.86  &4.24e-5 & 2.05& 5.90e-5 & 2.02&8.12e-5  &1.98& 8.49e-5& 2.13\\
& 9.82e-6 &1.93  &1.04e-5 & 2.02& 1.46e-5 & 2.01&2.04e-5  &1.99& 2.07e-5& 2.04\\
& 2.51e-6 &1.97  &2.59e-6 & 2.02& 3.65e-6 & 2.01&5.12e-6  &2.00& 5.14e-6& 2.01\\
\hline
\multirow{3}{*}{2}
&1.01e-4 & -   &2.00e-4 & -   & 2.43e-4 & -    & 2.84e-4& - &3.52e-4 & - \\
&1.60e-5 &2.66 &2.20e-5 & 3.19& 2.94e-5 & 3.04 &3.81e-5  &2.89& 4.00e-5& 3.14\\
&2.24e-6 &2.83 &2.61e-6 & 3.08& 3.61e-6 & 3.03 &4.90e-6  &2.96& 5.01e-6& 3.00\\
&2.96e-7 &2.92 &3.19e-7 & 3.03& 4.46e-7 & 3.01 &6.19e-7  &2.98& 6.24e-7& 3.00\\
&3.81e-8 &2.96 &3.94e-8 & 3.01& 5.55e-8 & 3.01 &7.78e-8  &2.99& 7.80e-8& 3.00\\
\hline
\multirow{3}{*}{3}
&5.44e-6& -     & 8.81e-6 & -   & 1.13e-5 & -   & 1.41e-5& - &1.58e-5 & - \\
&4.06e-7 &3.74  &5.05e-7 & 4.12 & 6.89e-7 & 4.03&9.19e-7 &3.94& 9.43e-7& 4.06\\
&2.75e-8 &3.88  &3.05e-8 & 4.05 & 4.25e-8 & 4.02&5.84e-8  &3.97& 5.92e-8& 3.99\\
&1.79e-9 &3.94  &1.88e-9 & 4.02 & 2.64e-9 & 4.01&3.68e-9  &3.99& 3.70e-9& 4.00\\
&1.14e-10 &3.97 &1.17e-10 & 4.01& 1.64e-10& 4.01 &2.31e-10  &4.00& 2.31e-10& 4.00\\
\hline
\end{tabular}
\end{center}
\caption{Error and convergence orders for problem \rev{3}.}
\label{tab:conv4}
\end{table}

\begin{table}
\begin{center}
    \renewcommand{\arraystretch}{1.3}
\begin{tabular}{|c|c|c|c|c|c|c|c|c|c|c|}
\hline
k & $e_{u}$ & order & $e_{q}$ & order & $e_{p}$ & order & $e_{r}$ & order & $e_{s}$ & order \\
\hline
\multirow{3}{*}{0}
&9.01e-3  & -   & 7.66e-3 & -    & 2.20e-2 & -    &5.53e-2& - &8.45e-2 & - \\
& 6.29e-3 &0.52 & 4.22e-3 & 0.86 & 1.22e-2 & 0.85 & 3.11e-2 &0.83& 5.83e-2& 0.54\\
& 3.86e-3 &0.70 & 2.47e-3 & 0.77 & 6.29e-3 & 0.96 & 1.67e-2 &0.89& 3.80e-2& 0.62\\
& 2.19e-3 &0.82 & 1.40e-3 & 0.82 & 3.19e-3 & 0.98 & 8.94e-3&0.90 & 2.22e-2& 0.77\\
& 1.18e-3 &0.89 & 7.53e-4 & 0.89 & 1.61e-3 & 0.98 & 4.70e-3&0.93 & 1.21e-2& 0.88\\
\hline
\multirow{3}{*}{1}
& 3.50e-4& -   &  7.95e-4 & - &  9.73e-4& -    &1.23e-3  & - & 2.50e-3& - \\
&9.46e-5 &1.89 &1.87e-4 & 2.09& 2.39e-4 & 2.03 & 3.20e-4  &1.94&  5.72e-4& 2.13\\
&2.48e-5 &1.93 &4.51e-5 & 2.05& 5.90e-5 & 2.02 & 8.12e-5  &1.98&  1.40e-4& 2.03\\
&6.37e-6 &1.96 &1.11e-5 & 2.03& 1.46e-5 & 2.01 &2.04e-5  &1.99& 3.50e-5& 2.00\\
& 1.61e-6 &1.98&2.74e-6 & 2.01& 3.65e-6 & 2.01  &5.12e-6  &2.00& 8.75e-6& 2.00\\
\hline
\multirow{3}{*}{2}
&8.02e-5 & -   & 2.09e-4 & -  & 2.45e-4 & -    & 2.93e-4& - &5.14e-4 & - \\
&1.12e-5 &2.84 &2.37e-5 & 3.14& 2.95e-5 & 3.05 &3.86e-5  &2.92& 6.43e-5& 3.00\\
&1.49e-6 &2.91 &2.82e-6 & 3.07& 3.61e-6 & 3.03 &4.92e-6  &2.97& 8.16e-6& 2.98\\
&1.92e-7 &2.95 &3.44e-7 & 3.04& 4.46e-7 & 3.01 &6.21e-7  &2.99& 1.02e-6& 3.00\\
&2.45e-8 &2.98 &4.25e-8 & 3.02& 5.55e-8 & 3.01 &7.79e-8  &2.99& 1.28e-7& 3.00\\
\hline
\multirow{3}{*}{3}
&4.17e-6 & -    & 9.20e-6 & -   & 1.13e-5 & -   & 1.43e-5& - &2.51e-5 & - \\
&2.79e-7 &3.90  &5.36e-7 & 4.10 & 6.90e-7 & 4.04&9.25e-7 &3.95& 1.59e-6& 3.98\\
&1.81e-8 &3.94  &3.24e-8 & 4.05 & 4.25e-8 & 4.02&5.86e-8  &3.98& 1.01e-7& 3.98\\
&1.16e-9 &3.97  &2.00e-9 & 4.03 & 2.64e-9 & 4.01&3.69e-9  &3.99& 6.30e-9& 4.00\\
&7.30e-11 &3.99 &1.23e-10 & 4.01& 1.64e-10& 4.01 &2.31e-9  &4.00& 3.94e-10& 4.00\\
\hline
\end{tabular}
\end{center}
\caption{Error and convergence orders for problem \rev{4}.}
\label{tab:conv5}
\end{table}

\section{Concluding Remarks}
\label{sec:conclusion}
In this paper, we introduce and analyze the first HDG method for fifth-order KdV type equations as well as fifth-order stationary equations. Our analysis show that the HDG approximate solutions have optimal convergence rates for linearized equations in the time-dependent case and superconvergence properties in the stationary case. Numerical results indicate that the method also has optimal convergence rates for nonlinear KdV equations. Our future work is to develop HDG methods for multidimensional KdV type equations and systems that involve third- or fifth-order derivatives.

\appendix

\section{Proof of Lemma \ref{lemma_energy}}
\label{sec:appendixA}

First, we derive a suitable error equation. Toward that end, we note that the exact solutions also satisfy \eqref{eq:scheme}. Subtracting \eqref{eq:scheme} from its counterpart \rev{satisfied} by the exact solution gives us
\[
\begin{cases}
(e_q, v) + (e_u, v_x) - \langle \widehat{e}_u, v \cdot n \rangle = 0,\\
(e_p, z) + (e_q, z_x) - \langle \widehat{e}_q, z \cdot n \rangle = 0,\\
(e_r, w) + (e_p, w_x) - \langle \widehat{e}_p, w \cdot n \rangle = 0,\\
(e_s, \phi) + (e_r, \phi_x) - \langle \widehat{e}_r, \phi \cdot n \rangle = 0,\\
(e_{u_t}, \psi) - (\alpha e_p + \beta e_s, \psi_x) + \langle \alpha \widehat{e}_p + \beta \widehat{e}_s, \psi \cdot n \rangle = 0.\\
\end{cases}
\]
We also have that
\[
\jmp{\widehat{e}_p \cdot n} = 0, \quad \jmp{\widehat{e}_r \cdot n} = 0, \quad \jmp{\widehat{e}_s \cdot n} = 0.
\]

If we define the projections of the numerical flux errors \rev{as}
\begin{subequations}
\begin{alignat*}{1}
\widehat{\varepsilon}_u & = \widehat{e}_u, \quad \widehat{\varepsilon}_q  = \widehat{e}_q, \quad \widehat{\varepsilon}_p^{\;-} = \widehat{e}_p^{\;-}, \\
\widehat{\varepsilon}_\omega^{\;+} & = \epsilon_\omega^+ + \tau_{\omega u}^+ (\widehat{\varepsilon}_u - \varepsilon_u^+) n^+ + \tau_{\omega q}^+ (\widehat{\varepsilon}_q - \varepsilon_q^+) n^+ \mbox{ for } \omega \in \{p, r, s\},\\
\widehat{\varepsilon}_\omega^{\;-} & = \epsilon_\omega^- + \tau_{\omega u}^- (\widehat{\varepsilon}_u - \varepsilon_u^-) n^-  + \tau_{\omega q}^- (\widehat{\varepsilon}_q - \varepsilon_q^-) n^- + \tau_{\omega p}^-(\widehat{\varepsilon}_p^{\;-} - \varepsilon_p^-) n^- \mbox{ for } \omega \in \{r, s\},\\
\end{alignat*}
\end{subequations}
it is easy to show that
\[
\widehat{\varepsilon}_\omega = \widehat{e}_\omega, \quad \mbox{ for } \omega \in \{u, q, p, r, s\}
\]
by using the definition of the projection $\varPi$. Realizing that $e_\omega = \delta_\omega + \varepsilon_\omega$ and using the definition of the projection $\varPi$, we can rewrite the error equation above as
\begin{subequations}\label{eq:error}
\begin{alignat}{1}
\label{eq:error1}
&({\varepsilon}_q, v) + (\delta_q, v) + (\varepsilon_u, v_x) - \langle \widehat{\varepsilon}_u, v \cdot n \rangle = 0,\\
\label{eq:error2}
&({\varepsilon}_p, z) + (\delta_p, z) + (\varepsilon_q, z_x) - \langle \widehat{\varepsilon}_q, z \cdot n \rangle = 0,\\
\label{eq:error3}
&({\varepsilon}_r, w) + (\delta_r, w) + (\varepsilon_p, w_x) - \langle \widehat{\varepsilon}_p, w \cdot n \rangle = 0,\\
\label{eq:error4}
&({\varepsilon}_s, \phi) + (\delta_s, \phi) + (\varepsilon_r, \phi_x) - \langle \widehat{\varepsilon}_r, \phi \cdot n \rangle = 0,\\
\label{eq:error5}
&({\varepsilon}_{u_t}, \psi) + (\delta_{u_t}, \psi) - (\alpha \varepsilon_p + \beta \varepsilon_s, \psi_x) + \langle \alpha \widehat{\varepsilon}_p + \beta \widehat{\varepsilon}_s, \psi \cdot n \rangle = 0,
\end{alignat}
\end{subequations}
and that
\begin{equation*}
\jmp{\widehat{\varepsilon}_p \cdot n} = 0, \quad \jmp{\widehat{\varepsilon}_r \cdot n} = 0, \quad \jmp{\widehat{\varepsilon}_s \cdot n} = 0.
\end{equation*}

To derive the energy identities, we will also need the time-derivatives of the equations \eqref{eq:error1} - \eqref{eq:error5}.
\begin{subequations}\label{eq:error_t}
\begin{alignat}{1}
\label{eq:error1_t}
&({\varepsilon}_{q_t}, v) + (\delta_{q_t}, v) + (\varepsilon_{u_t}, v_x) - \langle \widehat{\varepsilon}_{u_t}, v \cdot n \rangle = 0,\\
\label{eq:error2_t}
&({\varepsilon}_{p_t}, z) + (\delta_{p_t}, z) + (\varepsilon_{q_t}, z_x) - \langle \widehat{\varepsilon}_{q_t}, z \cdot n \rangle = 0,\\
\label{eq:error3_t}
&({\varepsilon}_{r_t}, w) + (\delta_{r_t}, w) + (\varepsilon_{p_t}, w_x) - \langle \widehat{\varepsilon}_{p_t}, w \cdot n \rangle = 0,\\
\label{eq:error4_t}
&({\varepsilon}_{s_t}, \phi) + (\delta_{s_t}, \phi) + (\varepsilon_{r_t}, \phi_x) - \langle \widehat{\varepsilon}_{r_t}, \phi \cdot n \rangle = 0,\\
\label{eq:error5_t}
&({\varepsilon}_{u_{tt}}, \psi) + (\delta_{u_{tt}}, \psi) - (\alpha \varepsilon_{p_t} + \beta \varepsilon_{s_t}, \psi_x) + \langle \alpha \widehat{\varepsilon}_{p_t} + \beta \widehat{\varepsilon}_{s_t}, \psi \cdot n \rangle = 0.
\end{alignat}
\end{subequations}

We are now ready to prove the six energy identities in Lemma \ref{lemma_energy}. The proofs are quite tedious, but very similar in nature. They also use similar techniques such as integration by parts, and
error equations \eqref{eq:error} and \eqref{eq:error_t}.
For these reasons, we are simply listing the first steps of deriving these six energy identities. The rest for proving each energy identity are simply algebraic manipulations.

Indeed, we take
\[
\psi = \varepsilon_u, \quad v = -(\alpha \varepsilon_p + \beta \varepsilon_s), \quad z = \alpha \varepsilon_q + \beta \varepsilon_r, \quad \phi = \beta \varepsilon_q, \quad w = -\beta \varepsilon_p
\]
in \eqref{eq:error1}--\eqref{eq:error5}, add them together to arrive at \eqref{eq:energyid1}, and
\[
v = \varepsilon_q \mbox{ in \eqref{eq:error1_t}}, \quad z = \varepsilon_{u_t} \mbox{ in \eqref{eq:error2}}
\]
for deriving \eqref{eq:energyid2}, and
\[
z = \varepsilon_p \mbox{ in \eqref{eq:error2_t}}, \quad w = \varepsilon_{q_t} \mbox{ in \eqref{eq:error3}},
\quad v = -\varepsilon_r \mbox{ in \eqref{eq:error1_t}}, \quad
\phi = -\varepsilon_{u_t} \mbox{ in \eqref{eq:error4}}
\]
for getting \eqref{eq:energyid3}, and
\[
w = \varepsilon_r \mbox{ in \eqref{eq:error3_t}}, \quad \phi = \varepsilon_{p_t} \mbox{ in \eqref{eq:error4}}, \quad z = -(\varepsilon_s +\frac{\alpha}{\beta} \varepsilon_p) \mbox{ in \eqref{eq:error2_t}},
\]
\[
 \psi = \frac{1}{\beta} \varepsilon_{q_t} \mbox{ in \eqref{eq:error5}}, \quad v = - \frac{1}{\beta} \varepsilon_{u_t} \mbox{ in \eqref{eq:error1_t} }
\]
for concluding \eqref{eq:energyid4}, and
\[
\phi = \varepsilon_s \mbox{ in \eqref{eq:error4_t}}, \quad \psi = -\frac{1}{\beta} \varepsilon_{r_t} \mbox{ in \eqref{eq:error5}}, \quad w = \frac{1}{\beta} \varepsilon_{u_t} \mbox{ in \eqref{eq:error3_t}}, 
\]
\[
w = -\frac{\alpha}{\beta} \varepsilon_{r_t} \mbox{ in \eqref{eq:error3}}, \quad v = \frac{1}{\beta} \varepsilon_{p_t} \mbox{ in \eqref{eq:error1_t},}\quad z = -\frac{1}{\beta} \varepsilon_{q_t} \mbox{ in \eqref{eq:error2_t}}
\]
to obtain \eqref{eq:energyid5}, and finally
\[
\psi = \varepsilon_{u_t} \mbox{ in \eqref{eq:error5_t}}, \quad v = -( \alpha \varepsilon_{p_t} + \beta \varepsilon_{s_t}) \mbox{ in \eqref{eq:error1_t}}, \quad z = \alpha \varepsilon_{q_t} +\beta \varepsilon_{r_t} \mbox{ in \eqref{eq:error2_t}}, 
\]
\[
\phi = \beta \varepsilon_{q_t} \mbox{ in \eqref{eq:error4_t},} \quad w = -\beta \varepsilon_{p_t} \mbox{ in \eqref{eq:error3_t}}
\]
for proving \eqref{eq:energyid6} to conclude the proof of the six energy identities.

To finish proving this lemma, we only need to establish that $T \ge 0$ which is the purpose of the next lemma.

\begin{lemma}
Let $T$ be defined as in Lemma \ref{lemma_energy}, and the stabilization functions are chosen so that the only nonzero ones are
\[
\tau_{su}^- = \tau_{su}^+ =1, \quad \tau_{rq}^- = \tau_{rq}^+ = -1,
\]
then for $\beta=-1$ we have $T \ge 0$.
\end{lemma}
\begin{proof}
Let's recall from Lemma \ref{lemma_energy} that $T$ contains all the boundary terms. For brevity, we define $\eta_\omega = \widehat{\varepsilon}_\omega - \varepsilon_\omega$ for $\omega \in \{u, q, p, r, s\}$, $T$ can then be rewritten as
\begin{alignat*}{1}
T = & \langle n, -\beta \eta_s \eta_u + \beta \eta_r \eta_q - \frac{\beta}{2} \eta_p^2 \rangle + \langle n, \eta_q \eta_{ut} \rangle + \langle n, \eta_p \eta_{qt} - \eta_r \eta_{ut} \rangle\\
& + \langle n, \eta_r \eta_{pt} - \eta_s \eta_{qt} - \frac{1}{2 \beta} \eta^2_{ut} \rangle
+ \langle n, \eta_s \eta_{rt} + \frac{1}{\beta} \eta_{ut} \eta_{pt} - \frac{1}{2 \beta} \eta^2_{qt} \rangle\\
&+ \langle n, -\beta \eta_{st} \eta_{ut} + \beta \eta_{rt} \eta_{qt} - \frac{\beta}{2} \eta^2_{pt} \rangle\\
:= & T^- + T^+,
\end{alignat*}
where $T^-$ contains those terms on $\partial\calT_h^-$ with $n^-=1$ and $T^+$ on $\partial\calT_h^+$ for $n^+=-1$. It suffices to simplify each of them and show that they are nonnegative.
\begin{alignat*}{1}
T^- = & \langle 1, \eta_u^2\rangle \left( -\beta \tau_{su}^- \right) + \langle 1, \eta_q^2\rangle \left( \beta \tau_{rq}^- \right) + \langle 1, \eta_p^2\rangle \left( -\frac{\beta}{2}\right) \\
& + \langle 1, \eta_{ut}^2\rangle \left( -\frac{1}{2 \beta} -\beta \tau_{su}^- \right)
+ \langle 1, \eta_{qt}^2\rangle \left( -\frac{1}{2 \beta} +\beta \tau_{rq}^- \right)
+ \langle 1, \eta_{pt}^2\rangle \left( -\frac{\beta}{2}\right)\\
& + \langle 1, \eta_u \eta_{ut} \rangle \left( -\tau_{ru}^- + \tau_{su}^- \tau_{ru}^-\right)
+ \langle 1, \eta_q \eta_{qt} \rangle \left( -\tau_{sq}^- + \tau_{sq}^- \tau_{rq}^-\right)
+ \langle 1, \eta_p \eta_{pt} \rangle \left( \tau_{rp}^- + \tau_{sp}^- \tau_{rp}^-\right)\\
& + \langle 1, \eta_u \eta_q \rangle \left( -\beta \tau_{sq}^- + \beta \tau_{ru}^- \right)
+ \langle 1, \eta_u \eta_p \rangle \left( -\beta \tau_{sp}^- \right)
+ \langle 1, \eta_p \eta_q \rangle \left( \beta \tau_{rp}^- \right)\\
& + \langle 1, \eta_q \eta_{ut} \rangle \left( 1 - \tau_{rq}^- + \tau_{sq}^- \tau_{ru}^- \right) + \langle 1, \eta_u \eta_{qt} \rangle \left( - \tau_{su}^- + \tau_{su}^- \tau_{rq}^- \right)\\
& + \langle 1, \eta_p \eta_{qt} \rangle \left( 1 - \tau_{sp}^- + \tau_{sp}^- \tau_{rq}^- \right) + \langle 1, \eta_q \eta_{pt} \rangle \left(  \tau_{rq}^- + \tau_{sq}^- \tau_{rp}^- \right)\\
& + \langle 1, \eta_u \eta_{pt} \rangle \left(  \tau_{ru}^- + \tau_{su}^- \tau_{rp}^- \right) + \langle 1, \eta_p \eta_{ut} \rangle \left( - \tau_{rp}^- + \tau_{sp}^- \tau_{ru}^- \right)\\
& + \langle 1, \eta_{ut} \eta_{pt} \rangle \left (\frac{1}{\beta} - \beta \tau_{sp}^- \right) + \langle 1, \eta_{qt} \eta_{ut} \rangle \left (-{\beta} \tau_{sq}^- + \beta \tau_{ru}^- \right) + \langle 1, \eta_{qt} \eta_{pt} \rangle \left (\beta \tau_{rp}^- \right).
\end{alignat*}
Therefore, if we choose $\tau_{su}^- =1, \quad \tau_{rq}^- = -1$, other $\tau^-_{\alpha \omega}$ be zero, and $\beta = -1$, we can complete the square for $T^-$
\[
T^- = \langle 1, (\eta_u - \eta_{qt})^2 + \frac{1}{2} (\eta_p + \eta_{qt})^2 + \frac{1}{2} (\eta_{pt} - \eta_q - \eta_{ut})^2 + \frac{1}{2}(\eta_q + \eta_{ut})^2 + \frac{1}{2} \eta_{ut}^2.
\]
We conclude that $T^- \ge 0$. For $T^+$, we have
\begin{alignat*}{1}
T^+ = & \langle 1, \eta_u^2 \rangle \left( -\beta \tau_{su}^+ + \frac{\beta}{2} {\tau_{pu}^+}^2\right) + \langle 1, \eta_q^2 \rangle \left( \beta \tau_{rq}^+ + \frac{\beta}{2} {\tau_{pq}^+}^2\right)\\
& + \langle 1, \eta_{ut}^2 \rangle \left(\frac{1}{2 \beta} + \frac{1}{\beta} \tau_{pu}^+ -\beta \tau_{su}^+ + \frac{\beta}{2} {\tau_{pu}^+}^2\right)
+ \langle 1, \eta_{qt}^2 \rangle \left(\frac{1}{2 \beta} + {\beta} \tau_{rq}^+ + \frac{\beta}{2} {\tau_{pq}^+}^2\right)\\
& + \langle 1, \eta_u \eta_{ut} \rangle \left( -\tau_{ru}^+ - \tau_{ru}^+ \tau_{pu}^+ - \tau_{su}^+ \tau_{ru}^+\right)
+ \langle 1, \eta_q \eta_{qt} \rangle \left( \tau_{pq}^+ - \tau_{rq}^+ \tau_{pq}^+ - \tau_{sq}^+ - \tau_{sq}^+ \tau_{rq}^+\right)\\
& + \langle 1, \eta_u \eta_{q} \rangle \left( -\beta \tau_{sq}^+  + \beta\tau_{ru}^+ + \beta \tau_{pu}^+ \tau_{pq}^+\right)
 + \langle 1, \eta_{ut} \eta_{q} \rangle \left(-1 -\tau_{rq}^+ - \tau_{rq}^+ \tau_{pu}^+ - \tau_{sq}^+ \tau_{ru}^+\right)\\
 & + \langle 1, \eta_{qt} \eta_{u} \rangle \left( \tau_{pu}^+ - \tau_{ru}^+ \tau_{pq}^+ -\tau_{su}^+ - \tau_{su}^+ \tau_{rq}^+\right)\\
& + \langle 1, \eta_{qt} \eta_{ut} \rangle \left( \frac{1}{\beta} \tau_{pq}^+ - \beta \tau_{sq}^+ + \beta \tau_{ru}^+ + \beta \tau_{pu}^+ \tau_{pq}^+\right).
\end{alignat*}
Therefore, if we choose $\tau_{su}^+ =1, \quad \tau_{rq}^+ = -1$, other $\tau^+_{\alpha \omega}$ be zero, then for $\beta=-1$, we see that
\[
T^+ = \langle 1, \eta_u^2 + \eta_{q}^2+\frac{1}{2} \eta_{ut}^2 + \frac{1}{2} \eta_{qt}^2  \rangle.
\]

\end{proof}

\section{Proof of Lemma \ref{lemma_approx}}

\label{sec:appendixB}

We prove it by arguing that the distance between $\varPi$ and the $L^2$-projection is optimal. Indeed, for any $\omega \in L^2(I_i)$, we denote its $L^2$-projection onto $P^k(I_i)$ by $(\omega)_k$. Then we have
\[
\delta_\omega = g_\omega + d_\omega, \quad \mbox{ where } g_\omega := \omega - (\omega)_k, \; d_\omega := (\omega)_k - \varPi \omega, \mbox{ for } \omega \in \{u, q, p, r, s\},
\]
and we only need to estimate $d_\omega$ as $g_\omega$ is of order $k+1$.

Obviously, $d_\omega \in P^k(I_i)$. From the definition of the projection $\varPi$, we have that $(d_\omega, v) =0$ for any $v \in P^{k-1}(I_i)$. This means that
\begin{equation}
\label{eq:disLk}
d_\omega = c_\omega \cdot L_k
\end{equation}
where $L_k$ is the scaled Legendre polynomial of degree $k$. Next, we rewrite the definition of the projection as
\[
\begin{cases}
d_s - \tau_{su}^- d_u - \tau_{sq}^- d_q - \tau_{sp}^- d_p = -\left ( g_s - \tau_{su}^- g_u - \tau_{sq}^- g_q - \tau_{sp}^- g_p \right) & \mbox{ on } x_i^-,\\
d_r - \tau_{ru}^- d_u - \tau_{rq}^- d_q - \tau_{rp}^- d_p = -\left ( g_r - \tau_{ru}^- g_u - \tau_{rq}^- g_q - \tau_{rp}^- g_p \right) & \mbox{ on } x_i^-,\\
d_p + \tau_{pu}^+ d_u + \tau_{pq}^+ d_q = -\left ( g_p + \tau_{pu}^+ g_u + \tau_{pq}^+ g_q \right) & \mbox{ on } x_{i-1}^+,\\
d_s + \tau_{su}^+ d_u + \tau_{sq}^+ d_q = -\left ( g_s + \tau_{su}^+ g_u + \tau_{sq}^+ g_q \right) & \mbox{ on } x_{i-1}^+,\\
d_r + \tau_{ru}^+ d_u + \tau_{rq}^+ d_q = -\left ( g_r + \tau_{ru}^+ g_u + \tau_{rq}^+ g_q \right) & \mbox{ on } x_{i-1}^+.
\end{cases}
\]
Applying \eqref{eq:disLk} and the fact that $L_k(x_i^-) = 1$ and $L_k(x_{i-1}^+) = (-1)^k$, this linear system can be written as
\[
{\mathbb A} \vec{c} = \vec{b}
\]
where
\[
{\mathbb A} =
\left(
\begin{tabular}{ccccc}
1 & 0 & $-\tau_{sp}^-$ & $-\tau_{sq}^-$ & $-\tau_{su}^-$\\
0 & 1 & $-\tau_{rp}^-$ & $-\tau_{rq}^-$ & $-\tau_{ru}^-$\\
0 & 0 & $(-1)^k$ & $(-1)^k \tau_{pq}^+$ & $(-1)^k \tau_{pu}^+$\\
$(-1)^k$ & 0 & 0 & $(-1)^k \tau_{sq}^+$ & $(-1)^k \tau_{su}^+$\\
0 & $(-1)^k$ & 0 & $(-1)^k \tau_{rq}^+$ & $(-1)^k \tau_{ru}^+$
\end{tabular}
\right)
\]
and
\[
\vec{c} =
\left(
\begin{tabular}{c}
$c_s$\\
$c_r$\\
$c_p$\\
$c_q$\\
$c_u$
\end{tabular}
\right),\quad
\vec{b} =
\left(
\begin{tabular}{c}
$b_1$\\
$b_2$\\
$b_3$\\
$b_4$\\
$b_5$
\end{tabular}
\right)
= -\left(
\begin{tabular}{c}
$g_s - \tau_{su}^- g_u - \tau_{sq}^- g_q + \tau_{sp}^- g_p$\\
$g_r - \tau_{ru}^- g_u - \tau_{rq}^- g_q + \tau_{rp}^- g_p$\\
$g_p + \tau_{pu}^+ g_u + \tau_{pq}^- g_q$\\
$g_s + \tau_{su}^+ g_u + \tau_{sq}^- g_q $\\
$g_r + \tau_{ru}^+ g_u + \tau_{rq}^- g_q$
\end{tabular}
\right).
\]
Basic row operation of ${\left( \displaystyle {\mathbb A} | \vec{b} \right)}$ leads to
\[
\left(
\begin{tabular}{ccccc|c}
1 & 0 & $-\tau_{sp}^-$ & $-\tau_{sq}^-$ & $-\tau_{su}^-$ & $b_1$\\
0 & 1 & $-\tau_{rp}^-$ & $-\tau_{rq}^-$ & $-\tau_{ru}^-$ & $b_2$\\
0 & 0 & 1 & $\tau_{pq}^+$ & $\tau_{pu}^+$ & $(-1)^k b_3$\\
0 & 0 & 0 & $\theta_q^s$ & $\theta_u^s$ & $\tilde{b}_4$\\
0 & 0 & 0 & $\theta_q^r$ & $\theta_u^r$ & $\tilde{b}_5$
\end{tabular}
\right),
\]
where $\theta_\omega^\alpha = \tau_{\alpha \omega}^- + \tau_{\alpha \omega}^+ - \tau_{\alpha p}^- \tau_{p \omega}^+$,
$\tilde{b}_4 = -b_1 + (-1)^k b_4 - (-1)^k \tau_{sp}^- b_3$, and $\tilde{b}_5 = -b_2 + (-1)^k b_5 - (-1)^k \tau_{rp}^- b_3$.
We can then invoke the Cramer's rule to solve for $\vec{c}$ when the linear system is uniquely solvable, that is when
\[
\Delta := \theta_q^s\theta_u^r - \theta_u^s\theta_q^r \neq 0.
\]
Under this condition, we have
\begin{alignat*}{1}
&c_q = \frac{\tilde{b}_4 \theta_u^r - \tilde{b}_5 \theta_u^s}{\Delta},\\
&c_u = \frac{\tilde{b}_5 \theta_q^s - \tilde{b}_4 \theta_q^r}{\Delta},\\
&c_p = (-1)^kb_3 - \tau_{pq}^+c_q - \tau_{pu}^+c_u,\\
&c_r = b_3+\tau_{rp}^-c_p + \tau_{rq}^-c_q + \tau_{ru}^-c_u,\\
&c_s = b_1+\tau_{sp}^-c_p + \tau_{sq}^-c_q + \tau_{su}^-c_u.
\end{alignat*}
We conclude that $\lVert c_\omega \rVert \le C h^{k+1}$ thanks to the fact that $\lVert g_\omega \rVert \le C h^{k+1}$ for any $\omega \in H^{k+1}(I_i)$. This completes the proof of Lemma \ref{lemma_approx}.

\begin{remark}
The choice of $\tau$ in the error analysis is that the only nonzero stabilization functions are $\tau_{su}^- = \tau_{su}^+ =1, \,\, \tau_{sq}^- = \tau_{sq}^+ = -1$. For this case, we have $\theta_q^s = 0, \, \theta_u^s = 2, \, \theta_q^r = -2$. This means that $\Delta = 4 \neq 0$.
\end{remark}

\section{Proof of Lemma \ref{lemma_init}}

\label{sec:appendixC}

The structure and techniques for proving the accuracy of the scheme \eqref{eq:schemest} are similar to \cite{ChenCockburnDongHDG16}. For completeness and brevity of this appendix, we are listing here all the results but only provide a sketch of each proof. The main results are superconvergence of the projection of the error and that of the numerical fluxes, as stated in the following two theorems.

\begin{theorem} [Error estimate for the linear steady-state fifth-order equation]
\label{thm:errorst}
Let $\omega$ be the solution to the steady-state equation \eqref{eq:equationst} with $F=0$, $\alpha=0$ and $\beta=-1$, $\varPi$ be the projection defined by \eqref{eq:proj}, and $\omega_h$ be the approximate solution given by \eqref{eq:schemest}, the projection of the error $\varepsilon_\omega$ satisfy
\begin{equation*}
\lVert \varepsilon_\omega \rVert \le C h^{k+1 + \min{\{k,1\}}} \quad\mbox{ for } \omega \in \{u, q, p, r, s\}
\end{equation*}
\end{theorem}

\begin{theorem} [Superconvergence of the numerical traces]
\label{thm:fluxst}
Under the assumption of Theorem \ref{thm:errorst},
the error of the numerical traces $\widehat{e}_\omega (x_i)$ satisfy
\begin{equation*}
|\widehat{e}_\omega (x_i)| \le C h^{2 k + 1} \quad \mbox{ for } \omega \in \{u, q, p, r, s\}.
\end{equation*}
\end{theorem}

The proof of \ref{thm:fluxst} is very similar to that for the third-order equation \cite{ChenCockburnDongHDG16}. We thus omit it and prove Theorem \ref{thm:errorst} before moving to fulfill the main charge of  this Appendix - establishing the main lemma \ref{lemma_init}.

\subsection{Proof for Theorem \ref{thm:errorst}.}

The proof adopts a duality argument, see also \cite{ChenCockburnDongHDG16}. For that purpose, we first introduce the dual problem
\begin{alignat}{1}
\label{eq:dualprob}
&\Psi_u' - \Psi_q \,=\, \eta_s, \quad \Psi_q' - \Psi_p\,=\, -\eta_r,\quad \\
&\Psi_p' - \Psi_r \,=\, \eta_p, \quad \Psi_r' - \Psi_s \,=\, -\eta_q, \quad \beta \Psi_s' + \gamma \Psi_u\,=\,\eta_u  \quad\quad {\textrm {in} }\;  \Omega,\nonumber\\
&\Psi_\omega(0, t) = \Psi_\omega(L, t) \quad \mbox{ for } \omega \in \{u, q, p, r, s\},\nonumber
\end{alignat}
where $\Psi'$ denotes $\Psi_x$ for any $\Psi$. Then we define a projection $\varPi^*: \left[ H^1(\calT_h) \right]^5 \longrightarrow \left[ W_h^k \right]^5$ that is dual of projection $\varPi$ defined by \eqref{eq:proj}:
For $(u, q, p, r, s) \in \left[ H^1(\calT_h) \right]^5$, we find in $\left[ W_h^k \right]^5$ the element $$(\varPi^* u, \varPi^* q, \varPi^* p, \varPi^* r, \varPi^* s)$$ whose restriction on element $K$ satisfies the following condition
\[
\begin{cases}
(\delta^* \Psi_u, v)_{I_i} = 0, \,\, (\delta^* \Psi_q, z)_{I_i} = 0, \,\, (\delta^* \Psi_p, w)_{I_i} = 0, \,\, & \\
(\delta^* \Psi_r, \phi)_{I_i} = 0, \,\, (\delta^* \Psi_s, \psi)_{I_i} = 0, \qquad\qquad& \\
\delta^*\Psi_p^- - \tau_{rp}^- \delta \Psi_q n + \tau_{sp}^- \delta \Psi_u n = 0 & \mbox{ at } x_{i}^-\\
\delta^*\Psi_r^- + \tau_{rq}^- \delta \Psi_q n - \tau_{sq}^- \delta \Psi_u n = 0 & \mbox{ at } x_{i}^-\\
\delta^*\Psi_s^- - \tau_{ru}^- \delta \Psi_q n + \tau_{su}^- \delta \Psi_u n = 0 & \mbox{ at } x_{i}^-\\
\delta^*\Psi_r^+ + \tau_{rq}^+ \delta \Psi_q n - \tau_{sq}^+ \delta \Psi_u n - \tau_{pq}^+ \delta \Psi_p n= 0 & \mbox{ at } x_{i-1}^+\\
\delta^*\Psi_s^+ - \tau_{ru}^+ \delta \Psi_q n + \tau_{su}^+ \delta \Psi_u n + \tau_{pu}^+ \delta \Psi_p n= 0 & \mbox{ at } x_{i-1}^+
\end{cases}
\]
for any $(v, z, w, \phi, \psi) \in P_{k-1}(I_i).$

We have a lemma on the regularity of the dual problem, and a lemma on the approximation properties of the dual projection $\varPi^*$.
\begin{lemma}[Regularity of the dual problem]
For the dual problem \eqref{eq:dualprob}, we have that
\[
\lVert \Psi_u \rVert_1 + \lVert \Psi_q \rVert_1  + \lVert \Psi_p \rVert_1  + \lVert \Psi_r \rVert_1 + \lVert \Psi_s \rVert_1 \le C_{\rm reg} \lVert \eta \rVert,
\]
where $\lVert \eta \rVert = \lVert \eta_u \rVert + \lVert \eta_q \rVert + \lVert \eta_p \rVert + \lVert \eta_r \rVert + \lVert \eta_s \rVert$.
\end{lemma}
\begin{lemma}
For the dual projection $\varPi^*$, we have that
\[
\lVert \delta^* \Psi_\omega \rVert \le C h \lVert \Psi_\omega \rVert_1, \quad \mbox{ for } \omega \in \{u, q, p, r, s\}.
\]
\end{lemma}
The proof of the second lemma is very similar to that of $\varPi$, therefore we only prove the first lemma here.
We integrate the last equation of the dual problem \eqref{eq:dualprob} over the spatial domain $\Omega$. Thanks to periodicity, we have
\[
\gamma \bar{\Psi}_u = \bar{\eta}_u, \quad \mbox{ where } \bar{\omega} := \int_\Omega \omega dx \quad \mbox{ for any } \omega \in L^2(\Omega),
\]
which implies that $\lVert \bar{\Psi}_u \rVert = \lVert \gamma^{-1} \bar{\eta}_u \rVert \le \frac{C}{\gamma} \lVert \eta_u \rVert$.
By Poincare inequality and the first equation of \eqref{eq:dualprob}, we have $\lVert \Psi_u - \bar{\Psi}_u \rVert \le C \lVert \Psi_u' \rVert \le \lVert \Psi_q \rVert + \lVert \eta_s \rVert$ which means
\[
\lVert \Psi_u \rVert \le \lVert \Psi_q \rVert + \lVert \eta_s \rVert + \frac{C}{\gamma} \lVert \eta_u \rVert.
\]
Similarly, we have
\begin{alignat*}{1}
\lVert \Psi_q \rVert & \le \lVert \Psi_p \rVert + \lVert \eta_r \rVert + C \lVert \eta_s \rVert\\
\lVert \Psi_p \rVert & \le \lVert \Psi_r \rVert + \lVert \eta_p \rVert + C \lVert \eta_r \rVert\\
\lVert \Psi_r \rVert & \le \lVert \Psi_s \rVert + \lVert \eta_q \rVert + C \lVert \eta_p \rVert\\
\lVert \Psi_s \rVert & \le \frac{1}{|\beta|} \left( \gamma \lVert \Psi_u \rVert + \lVert \eta_u \rVert \right )+ C \lVert \eta_q \rVert.
\end{alignat*}
We conclude that
\begin{equation}
\label{eq:omegabyu}
\lVert \Psi_\omega \rVert \le \frac{\gamma}{|\beta|} \lVert \Psi_u \rVert + C \lVert \eta \rVert \quad \mbox{ for } \omega \in \{u, q, p, r, s\}.
\end{equation}

Next, we multiply the five equations of \eqref{eq:dualprob} by $(\beta \Psi_s, -\beta \Psi_r, \beta \Psi_p, -\beta \Psi_q, \Psi_u)$ respectively, integrate them on $\Omega$, and add all five together to obtain
\begin{alignat*}{1}
\gamma\|\Psi_u \|^2 = & \beta \int_\Omega \left( \Psi_s \eta_s + \Psi_r \eta_r + \Psi_p \eta_p + \Psi_q \eta_q + \Psi_u \eta_u \right) dx\\
\le & \frac{| \beta | \epsilon }{2} \left( \lVert \Psi_s \rVert^2 + \lVert \Psi_r \rVert^2 + \lVert \Psi_p \rVert^2 +\lVert \Psi_q \rVert^2 + \lVert \Psi_u \rVert^2 \right) + \frac{|\beta| \epsilon^{-1}}{2} \lVert \eta \rVert^2.
\end{alignat*}
We then invoke \eqref{eq:omegabyu} and let $\epsilon$ be small enough to arrive at $\lVert \Psi_u \rVert \le C \lVert \eta \rVert$ which implies, due to \eqref{eq:omegabyu}, that
\[
\lVert \Psi_\omega \rVert \le C \lVert \eta \rVert  \quad \mbox{ for } \omega \in \{u, q, p, r, s\}.
\]
Finally to conclude the proof of this lemma, we note from the first equation of \eqref{eq:dualprob} that
\[
| \Psi_u |_1 \le \lVert \Psi_q \rVert + \lVert \eta_s \rVert \le C \lVert \eta \rVert.
\]
Similarly, $| \Psi_\omega |_1 \le C \lVert \eta \rVert$ for $\omega \in \{q, p, r, s\}$.

With these two lemmas proven, we now go back to prove Theorem \ref{thm:errorst}. The first \rev{ingredient} we need is the error equation that is similar to those for the time-dependent case \eqref{eq:error1} - \eqref{eq:error5}. For simplicity,we assume that $\beta = -1$.
\begin{equation}\label{eq:errorst}
\begin{split}
&({\varepsilon}_q, v) + (\delta_q, v) + (\varepsilon_u, v') - \langle \widehat{\varepsilon}_u, v \cdot n \rangle = 0,\\
&({\varepsilon}_p, z) + (\delta_p, z) + (\varepsilon_q, z') - \langle \widehat{\varepsilon}_q, z \cdot n \rangle = 0,\\
&({\varepsilon}_r, w) + (\delta_r, w) + (\varepsilon_p, w') - \langle \widehat{\varepsilon}_p, w \cdot n \rangle = 0,\\
&({\varepsilon}_s, \phi) + (\delta_s, \phi) + (\varepsilon_r, \phi') - \langle \widehat{\varepsilon}_r, \phi \cdot n \rangle = 0,\\
&(\gamma {\varepsilon}_{u}, \psi) + (\gamma \delta_{u}, \psi) + (\varepsilon_s, \psi') - \langle \widehat{\varepsilon}_s, \psi \cdot n \rangle = 0.
\end{split}
\end{equation}

{\bf The proof of theorem \ref{thm:errorst}} is through the classical duality argument. Indeed, we define
\[
\lVert \varepsilon \rVert^2 := (\varepsilon_u, \eta_u) + (\varepsilon_q, \eta_q) + (\varepsilon_p, \eta_p) + (\varepsilon_r, \eta_r) + (\varepsilon_s, \eta_s)
\]
Using  the dual problem \eqref{eq:dualprob}, we can rewrite it as $\lVert \varepsilon \rVert^2 = \Theta_I + \Theta_{II}$ where
\begin{alignat*}{1}
\Theta_{I} & = (\varepsilon_u, \Psi_s') - (\varepsilon_q, \Psi_r') + (\varepsilon_p, \Psi_p') - (\varepsilon_r, \Psi_q') + (\varepsilon_s, \Psi_u')\\
\Theta_{II} & = (\varepsilon_u, \gamma \Psi_u) + (\varepsilon_q, \Psi_s) - (\varepsilon_p, \Psi_r) + (\varepsilon_r, \Psi_p) - (\varepsilon_s, \Psi_q)
\end{alignat*}
Invoking $$(\varepsilon_\omega, \Psi_v') = (\varepsilon_\omega, \varPi^* \Psi_v') + (\varepsilon_\omega, \delta^* \Psi_v') = (\varepsilon_\omega, \varPi \Psi_v') + \langle \varepsilon_\omega, \delta^* \Psi_v \cdot n \rangle$$ by using integration by parts and the definition of $\varPi^*$, we can further rewrite $\Theta_I$ as $\Theta_I = I_1 + I_2$ with
\begin{alignat*}{1}
I_1 & = (\varepsilon_u, \varPi^* \Psi_s') - (\varepsilon_q, \varPi^* \Psi_r') + (\varepsilon_p, \varPi^* \Psi_p') - (\varepsilon_r, \varPi^* \Psi_q') + (\varepsilon_s, \varPi^* \Psi_u')\\
I_2 & = \langle \varepsilon_u, \delta^* \Psi_s \cdot n \rangle - \langle \varepsilon_q, \delta^* \Psi_r \cdot n \rangle + \langle \varepsilon_p, \delta^* \Psi_p \cdot n \rangle - \langle \varepsilon_r, \delta^* \Psi_q \cdot n \rangle + \langle \varepsilon_s, \delta^* \Psi_u \cdot n \rangle.
\end{alignat*}
We now apply the error equations \eqref{eq:errorst} 
with $v = \varPi^* \Psi_s, z = \varPi^* \Psi_r, w = \varPi^* \Psi_p, \phi = \varPi^* \Psi_q, \psi = \varPi^* \Psi_u$ to rewrite $I_1$ as
\begin{alignat*}{1}
I_1 = & -(\varepsilon_q, \varPi^* \Psi_s) + (\varepsilon_p, \varPi^* \Psi_r) - (\varepsilon_r, \varPi^* \Psi_p) + (\varepsilon_s, \varPi^* \Psi_q) - (\gamma \varepsilon_u, \varPi^* \Psi_u)\\
& -(\delta_q, \varPi^* \Psi_s) + (\delta_p, \varPi^* \Psi_r) - (\delta_r, \varPi^* \Psi_p) + (\delta_s, \varPi^* \Psi_q) - (\gamma \delta_u, \varPi^* \Psi_u)\\
& + \langle \widehat{\varepsilon}_u, \varPi^* \Psi_s \cdot n \rangle - \langle \widehat{\varepsilon}_q, \varPi^* \Psi_r \cdot n \rangle + \langle \widehat{\varepsilon}_p, \varPi^* \Psi_p \cdot n \rangle - \langle \widehat{\varepsilon}_r, \varPi^* \Psi_q \cdot n \rangle + \langle \widehat{\varepsilon}_s, \varPi^* \Psi_u \cdot n \rangle.
\end{alignat*}
Adding $I_1$, $I_2$, and $\Theta_{II}$, keeping in mind that $\widehat{\varepsilon}_\omega$ and $\Psi_\omega$ are single-valued and periodic for $\omega \in \{u, q, p, r, s\}$, and performing simple manipulations, we have
\[
\lVert \varepsilon \rVert^2 = \theta_1 + \theta_2 + \theta_3,
\]
where
\begin{alignat*}{1}
\theta_1 = & (\gamma \varepsilon_u, \delta^* \Psi_u) + (\varepsilon_q, \delta^* \Psi_s) - (\varepsilon_p, \delta^* \Psi_r) + (\varepsilon_r, \delta^* \Psi_p) - (\varepsilon_s, \delta^* \Psi_q),\\
\theta_2 = & -(\delta_q, \varPi^* \Psi_s) + (\delta_p, \varPi^* \Psi_r) - (\delta_r, \varPi^* \Psi_p) + (\delta_s, \varPi^* \Psi_q) - (\gamma \delta_u, \varPi^* \Psi_u),\\
\theta_3 =&  - \langle \delta^* \Psi_s, (\widehat{\varepsilon}_u - \varepsilon_u)\cdot n \rangle + \langle \delta^* \Psi_r, (\widehat{\varepsilon}_q - \varepsilon_q)\cdot n \rangle - \langle \delta^* \Psi_p, (\widehat{\varepsilon}_p - \varepsilon_p)\cdot n \rangle \\
&+ \langle \delta^* \Psi_q, (\widehat{\varepsilon}_r - \varepsilon_r)\cdot n \rangle - \langle \delta^* \Psi_u, (\widehat{\varepsilon}_s - \varepsilon_s)\cdot n \rangle.
\end{alignat*}
We next estimate $\theta_i$'s one by one. We bound $\theta_1$ using the approximation property of $\varPi^*$.
\[
|\theta_1| \le C h \left( \lVert \varepsilon_u \rVert + \lVert \varepsilon_q \rVert + \lVert \varepsilon_p \rVert  + \lVert \varepsilon_r \rVert  + \lVert \varepsilon_s \rVert \right) \lVert \eta \rVert.
\]
We estimate $\theta_2$ by realizing that $(\delta_v, \varPi^* \Psi_w) = (\delta_v, -\delta^* \Psi + \delta^{k-1} \Psi_w)$. Here $\delta^{k-1} \Psi_w := \Psi_w - (\Psi_w)_{k-1}$ with $\lVert \delta^{k-1} \Psi_\omega \rVert \le C h^{\min{\{1,k\}}} \lVert \eta \rVert$ for $\omega \in \{u, q, p, r, s\}$. Due to this bound and the properties of the projections $\varPi$ and $\varPi^*$, we have
\[
| \theta_2 | \le C h^{k+1+\min{\{1,k\}}} \lVert \eta \rVert.
\]
Lastly, we realize that $\theta_3 = 0$ thanks to the definition of $\widehat{\varepsilon}_s$, $\widehat{\varepsilon}_r$, $\widehat{\varepsilon}_p^-$ and $\varPi^*$.

To conclude, we take $\eta_\omega = \varepsilon_\omega$ for $\omega \in \{u, q, p, r, s\}$ to obtain
\[
\lVert \varepsilon \rVert^2 \le |\theta_1 | + |\theta_2| \le Ch \lVert \varepsilon \rVert^2 + C h^{k+1+\min{\{k,1\}}} \lVert \varepsilon \rVert.
\]
When $h$ is small enough, we have that $\lVert \varepsilon \rVert \le C h^{k+1+\min{\{k,1\}}}$.

\subsection{Proof of the Lemma \ref{lemma_init}.}

We take $t = 0$ and $\psi = \varepsilon_{ut}(0)$ in \eqref{eq:error5} to obtain
\[
(\varepsilon_{ut}(0), \varepsilon_{ut}(0)) + (\delta_{ut}(0), \varepsilon_{ut}(0)) - (\beta \varepsilon_s(0), \varepsilon_{ut}(0)_x) + \langle \beta \widehat{\varepsilon}_s(0), \varepsilon_{ut}(0) \cdot n \rangle = 0
\]
By Cauchy inequality, inverse inequality, and the trace inequality, we have
\[
\lVert \varepsilon_{ut}(0) \rVert^2 \le C \lVert \delta_{ut}(0) \rVert \lVert \varepsilon_{ut}(0) \rVert + C h^{-1} \lVert \varepsilon_{s}(0) \rVert \lVert \varepsilon_{ut}(0) \rVert + C h^{-1/2} \lVert \widehat{\varepsilon}_{s}(0) \rVert_{\partial \Omega_h} \lVert \varepsilon_{ut}(0) \rVert
\]
which means
\begin{equation}
\label{eq:fm2estimate}
\lVert \varepsilon_{ut}(0) \rVert^2 \le C \lVert \delta_{ut}(0) \rVert^2 + C h^{-2} \lVert \varepsilon_{s}(0) \rVert^2 + C h^{-1} \lVert \widehat{\varepsilon}_{s}(0) \rVert_{\partial \Omega_h}.
\end{equation}
The estimates in Theorems \ref{thm:errorst} and \ref{thm:fluxst} imply that
\[
\lVert \varepsilon_s(0) \rVert \le C h^{k+2}, \quad \lVert \widehat{\varepsilon}_s(0) \rVert \le C h^{2k+1}.
\]
\rev{Moreover}, the properties of the projection $\varPi$ states that $\lVert \delta_{ut}(0) \rVert \le C h^{k+1}$.
Plugging these three estimates into \eqref{eq:fm2estimate} gives
\[
\lVert \varepsilon_{ut}(0) \rVert \le C h^{k+1},
\]
which concludes the proof of the lemma.

\end{document}